\documentclass[11pt]{amsart}
\usepackage{amsmath}
\usepackage{amsthm}
\usepackage{soul}
\usepackage{amsfonts}
\usepackage{color}
\usepackage{amssymb}
\usepackage{amscd}%%%added by Jan for V. Kac 4/23/07
\numberwithin{equation}{section}
\newtheorem{theorem}{Theorem}[section]
\newtheorem{theorema}{Theorem}
\newtheorem{cor}[theorem]{Corollary}
\newtheorem{remark}[theorem]{Remark}
\newtheorem{proposition}[theorem]{Proposition}
\newtheorem{lemma}[theorem]{Lemma}

\newtheorem{problem}[theorem]{Problem}
\theoremstyle{definition}
\newtheorem{defi}{Definition}[section]
\newtheorem{rem}[defi]{Remark}

\newcommand\half{\frac{1}{2}}

\newcommand\be{\beta}

\newcommand\g{\mathfrak g}

\newcommand\h{\mathfrak h}

\newcommand\D{\Delta}
\renewcommand\l{\lambda}
\newcommand\Dp{\Delta^+}

\renewcommand\d{\delta}

\renewcommand\a{\alpha}
\renewcommand\aa{\mathfrak a}

\renewcommand\k{\mathfrak k}

\newcommand{\Z}{\mathbb Z}

\newcommand\ganz{\mathbb Z}

\renewcommand\aa{\mathfrak a}

\newcommand\e{\epsilon}
\newcommand\C{\mathbb C}

\newcommand{\ZZ}{\mathbb{Z}}

\newcommand{\sdim}{\text{\rm sdim}}
\newcommand{\vac}{{\bf 1}}
\newcommand{\spo}{{spo(2|3)}}
\newcommand{\bea}{\begin{eqnarray}}
\newcommand{\eea}{\end{eqnarray}}
\begin{document}
\title[Conformal embeddings of  superalgebras]{Conformal embeddings in affine vertex superalgebras}
\author[Adamovi\'c, M\"oseneder, Papi, Per\v{s}e]{Dra{\v z}en~Adamovi\'c}
%\author[]{Victor~G. Kac}
\author[]{Pierluigi M\"oseneder Frajria}
\author[]{Paolo  Papi}
\author[]{Ozren  Per\v{s}e}
\begin{abstract} 
	This paper is a natural continuation of our previous work on conformal embeddings of vertex algebras \cite{AKMPP}, \cite{AKMPP-JA}, \cite{AKMPP-JJM}.  Here we consider conformal embeddings in simple affine vertex superalgebra  $V_k(\g)$ where  $\g=\g_{\bar 0}\oplus \g_{\bar 1}$ is a basic classical simple Lie superalgebra. Let $\mathcal V_k (\g_{\bar 0})$ be the subalgebra of  $V_k(\g)$  generated by $\g_{\bar 0}$. We first  classify all  levels $k$ for which the  embedding $\mathcal V_k (\g_{\bar 0})$ in $V_k(\g)$ is conformal. Next we prove that, for a large family of such conformal levels, $V_k(\g)$ is a  completely  reducible $\mathcal V_k (\g_{\bar 0})$--module and obtain decomposition  rules. Proofs are based on fusion rules arguments and on the representation theory of certain affine vertex algebras. The most interesting case is the decomposition of $V_{-2} (osp(2n +8 \vert 2n))$ as a finite, non simple current extension of $V_{-2} (D_{n+4}) \otimes V_1 (C_n)$. This decomposition uses our previous work \cite{AKMPP-IMRN} on the representation theory of $V_{-2} (D_{n+4})$.\par
	We also study conformal embeddings $gl(n \vert m) \hookrightarrow sl(n+1 \vert m)$ and in most cases we obtain decomposition rules.
\end{abstract}
\maketitle
\tableofcontents
 \section{Introduction}
 
 This paper is a natural continuation of our previous work on conformal embeddings of vertex algebras \cite{AKMPP}, \cite{AKMPP-JA}, \cite{AKMPP-JJM}. We are focused on embeddings of  affine vertex algebras into vertex superalgebras $V_k(\g)$, where  $\g=\g_{\bar 0}\oplus \g_{\bar 1}$ is a basic classical simple Lie superalgebra. Recall that if $V$ is a VOA and $W$ is a subVOA, the embedding $W\subset V$ as vertex algebras  is said to be {\it conformal} if both VOAs share the same conformal vector. It is difficult to classify all conformal embeddings in $V_k(\g)$, so we confine ourselves to deal with  a simpler problem:
\begin{problem}[Classification problem] \label{p1} Classify levels $k$ such that the affine vertex subalgebra $\mathcal  V_k(\g_{\bar 0})$  generated by $\g_{\bar 0} \subset \g$ is conformally embedded into $V_k(\g)$.
 \end{problem}
 We completely solve this problem: see Theorem \ref{cf}. Classification of conformal levels (i.e., levels solving Problem \ref{p1})  is also important as a motivation for studying the representation theory of affine vertex algebras at conformal levels. In many cases such conformal levels  have also appeared in our earlier works on conformal and collapsing levels for affine $\mathcal W$--algebras \cite{AKMPP-JA}, \cite{AKMPP-IMRN}.
 
 After classification of conformal levels, we are ready to consider the next  important problem:
 
 \begin{problem}[Simplicity problem]  \label{pr-2}Assume that $k$ is a conformal level.
   Determine the structure of the subalgebra $\mathcal V_k(\g_{\bar 0})$. In particular, determine when $\mathcal V_k(\g_{\bar 0})$ is simple. \end{problem}
 In the current paper we  focus on proving  simplicity of  $\mathcal V_k(\g_{\bar 0})$ in several interesting cases.
    
    \par
    
 The proof of simplicity is very natural  when a free-field realization of $V_k(\g)$ is available.  Here are examples of such cases:
$$\g= sl(m \vert n),\quad \g\text{ is of type $B(m,n)$, $D(m,n)$ or $C(n)$,  $k=1$}.
$$
 
The  general simplicity problem, when a realization is missing, is usually very delicate. We can solve it  in the following cases: \begin{itemize}
 \item $\g= spo(2 \vert 3)$, $k=-3/4$.
 \item $\g=F(4)$, $k=1$.
 \item $\g = G(3)$, $k=1$.
 \item $\g = osp(2n +8| 2n)$, $k=-2$.
 \end{itemize}

The last case $\g = osp(2n +8| 2n)$, $k=-2$ is very interesting since the subalgebra 
$\mathcal V_{-2} (\g_{\bar 0}) \cong V_{-2} (D_{n+4} ) \otimes V_1 (C_n)$. 
Here we explore the representation theory of  the simple vertex algebra $V_{-2} (D_{n+4} )$ developed in \cite{AKMPP-IMRN}, which gives that $V_{-2}(\g)$ is semi--simple as $\mathcal V_{-2} (\g_{\bar 0})$--modules. 

There are interesting cases when $\mathcal V_k(\g_{\bar 0})$ is not simple (cf. Remark \ref{D21}). In our paper \cite{AKMPP-Selecta} we detected similar cases for non-regular conformal embeddings. It turns out that analysis of these cases requires different techniques, and we plan to investigate them in our future research.
\vskip10pt
The simplicity problem is related with the next natural problem:
  \begin{problem}[Decomposition problem]  \label{pr-3} Assume that $k$ is a conformal level.
Describe the structure of  $V_k(\g)$ as a  $\mathcal V_k(\g_{\bar 0})$--module.   \end{problem}

   In the  cases dealt with in this paper, we are able to solve  both the simplicity and the decomposition problem using what we call fusion rules argument. 
By 

this we mean  the following: suppose that $W\subset V$ is an embedding of vertex algebras. Let $\mathcal M$ be a collection of $W$--submodules of $V$ that generates $V$ as a vertex algebra. Then the structure of $span(\mathcal M)$ under the dot product (cf. \eqref{dot}) in the set of all $W$--submodules gives information about the structure of $V$ as a $W$--module.
If the embedding is conformal then there are constraints that allow in many cases to recover the structure of $span(\mathcal M)$ and solve the simplicity and decomposition problems. 

Since  we  study decomposition rules only in the cases when $\mathcal V_k(\g_{\bar 0})$ is a simple vertex algebra,  the decomposition of $V_k(\g)$ is naturally related with the extensions of  the simple vertex algebra $\mathcal V_{k}(\g_{\bar 0})$.  \par

 When $\g$ is even, $\k$ is a subalgebra of $\g$,   and  $V_k(\g)$ is an extension of simple current type of   the conformal subalgebra $\mathcal V_k(\k)$, we were able (cf.  \cite{AKMPP}, \cite{AKMPP-JJM}, \cite{AKMPP-Selecta}) to  get  explicit decomposition rules  without knowing precisely the  fusion rules for $\mathcal V_{k}(\k)$--modules. We can apply such methods here to  obtain  decomposition formulas when $V_k(\g)$ is a simple current extension of $\mathcal V_{k}(\g_{\bar 0})$. These decompositions are presented in Subsection \ref{easy} (see e.g. Proposition \ref{justapply}). Interestingly, in many such cases we also have explicit realizations.
 
The previous  analysis  does not apply to  $\g=psl(n \vert n)$ and  in this case $V_1(\g)$ does not have explicit realization. But using fusion rules for $V_{-1}(sl(n))$ from \cite{AP2} we obtain the following result (see Theorem \ref{psl(n,n)-dec}).
   \begin{theorema} For $n \ge 3$, $V_1(psl(n \vert n))$ is a simple current extension of $V_1(sl(n)) \otimes V_{-1}(sl(n))$; the related  decomposition is given in \eqref{decpslnn}.  
   \end{theorema}
   
  Next we consider some cases   when $V_k(\g)$ is not a  simple  current extension of $\mathcal V_{k}(\g_{\bar 0})$.   The next theorem sums up the results proven in  Proposition \ref{osp(3|2)-dec}, Theorem \ref{conj-f4} and Theorem \ref{G2inG(3)}.  \begin{theorema} Assume that we are in the following cases of conformal embeddings
    \begin{itemize}
 \item $\g= spo(2 \vert 3)$, $k=-3/4$.
 \item $\g=F(4)$, $k=1$.
 \item $\g = G(3)$, $k=1$.
 \end{itemize}
   Then $V_k(\g)$ is a  finite, non simple current extension of  $\mathcal V_k(\g_{\bar 0})$
   \end{theorema}
     These cases (among others) have been previously  studied by T. Creutzig in    \cite{C-RIMS} using the extension theory of the vertex algebra $\mathcal V_k(\g_{\bar 0})$ and tensor category arguments. He impressively identifies the larger vertex algebra $V_k(\g)$ among all possible extensions of $\mathcal V_k(\g_{\bar 0})$. We  present a different (and more elementary) proof which uses only  some affine fusion rules.  In particular, we first prove that $\mathcal V_k(\g_{\bar 0})$ is simple. In our cases, this directly implies that $\mathcal V_k(\g_{\bar 0})$ is semi-simple in the category $KL_k$ (see \cite[\S 3]{AKMPP-IMRN}) and therefore $V_k(\g)$ is a completely reducible as $\mathcal V_k(\g_{\bar 0})$--module. To obtain the  precise decomposition we apply $\mathcal V_k(\g_{\bar 0})$--fusion rules.
    
    Our fusion rules method can be applied beyond the  affine vertex algebra setting. As an example, we present in Section \ref{FFR} a new proof of simplicity of the free-field realization of $V_1(osp(n|m))$ (cf. \cite{Kw}) and corresponding decomposition of the Fock space.
As a consequence, this also gives a new proof of  the simplicity of the realization of  $V_{-1/2} (sp(2n))$ from \cite{FF}.

     In Section \ref{osp-conformal} we deal with  $\g = osp(2n +8| 2n)$, $k=-2$.   We have the following result (cf. Theorems \ref{dosp}, \ref{ddosp}):
     \begin{theorema} \label{slutnja-uvod} Assume that $n \ge 1$.
     		We have the following decomposition
     	$$ V_{-2}(osp(2n +8| 2n)) = \bigoplus_{i = 0} ^{n} L_{-2} (i\omega_1) \otimes L_1(\omega_i). $$
     \end{theorema}

   \vskip 5mm
   
   {\bf Acknowledgments:} 
 We would like to thank Thomas  Creutzig and Victor Kac  for valuable  discussions. 
We also thank the referee for his/her careful reading of our paper.

Dra{\v z}en~Adamovi\'c  and Ozren  Per\v{s}e are  partially supported by  the
QuantiXLie Centre of Excellence, a project cofinanced
by the Croatian Government and European Union
through the European Regional Development Fund - the
Competitiveness and Cohesion Operational Programme
(KK.01.1.1.01.0004)
 
 \section{Setup and preliminary results}
\subsection{Notation}

Let  $\g=\g_{\bar 0}\oplus \g_{\bar 1}$ be a basic classical simple Lie superalgebra. Recall  that among all simple finite-dimensional Lie superalgebras it is characterized by the properties that its even part $\g_{\bar 0}$ is reductive and that it admits a non-degenerate invariant supersymmetric bilinear form $( \cdot | \cdot)$. 
 A complete list of basic classical simple Lie superalgebras  consists of simple 
finite-dimensional Lie algebras and the Lie superalgebras $sl(m|n)\ (m,n\ge 1, m\ne n)$, $psl(m|m)\ (m\ge 2)$, $osp(m|n)=spo(n|m)\ (m\geq 1, n\ge 2 \text{ even})$, $D(2,1;a)\ (a\in\C,\,a\ne 0,-1)$, $F(4)$, $G(3)$. Recall that $sl(2|1)$and $osp(2|2)$ are isomorphic. Also, the Lie superalgebras $D(2,1;a)$ and $D(2,1;a')$ are isomorphic if and only if $a, a'$  lie on  the same orbit of the group generated by the transformations $a\mapsto a^{-1}$ and $a\mapsto -1-a$, and $D(2,1;1)= osp(4|2)$. See \cite{KacLie} for details.\par
Choose a Cartan subalgebra $\h$ for $\g_{\bar 0}$  and let $\D$ be the set of roots. If $\D^+$ is a set of positive roots, then we let $\Pi$ be the corresponding set of simple roots. If $\g$ is not an even Lie algebra, we choose as $\D^+$ the distinguished set of positive roots (i.e. $\Pi$ has the minimal number of odd roots) from Table 6.1 of \cite{Kw}.
\par
We normalize the form $( \cdot | \cdot)$ as follows: if $\g$ is an even simple Lie algebra then require $(\theta|\theta)=2$ (where $\theta$ is the highest root of $\g$). If $\g$ is not even, then we let $(\cdot|\cdot)$ be the form described explicitly in Table 6.1 of of \cite{Kw}. Let $C_\g$ be the Casimir element of $\g$ and let $2h^\vee$ be the eigenvalue of its action on $\g$.

Let $k\in \C$ be non-critical, i.e. $k\ne - h^\vee$. We let $V^k(\g)$,  $V_k(\g)$ denote, respectively, the universal and the simple affine vertex algebra (see \cite[\S\ 4.7 and Example 4.9b]{KacV}). Note that the definition of $V^k(\g),V_k(\g)$ depends on the choice of $(\cdot|\cdot)$.

Let  $\g^0$ be an equal rank basic classical  subsuperalgebra of $\g$ such that the restriction of $( \cdot | \cdot)$ is nondegenerate.
We further assume that 
 $\g^0$  decomposes as  $\g^0=\g^0_0\oplus\dots\oplus \g^0_s$ with $\g^0_0$ even abelian and $\g^0_i$ basic classical simple ideals for $i>0$. A remarkable example of such a situation is the case $\g^0=\g_{\bar 0}$.
 If $\nu\in \h^*$, we set $\nu^j=\nu_{|\h\cap\g^0_j}$. 
For a simple basic classical Lie superalgebra $\aa$, we let $V_{\aa}(\mu)$ denote the irreducible finite dimensional representation  of $\aa$ of highest weight $\mu$. If $U$ is an irreducible finite dimensional representation of $\aa$, we let $L_{\aa,k}(U)$ be the irreducible representation of $V^k(\aa)$ with top component $U$. We simply write $L_{\aa}(\mu)$ or $L_{k}(\mu)$  for   $L_{\aa,k}(V_{\aa}(\mu))$.

Let $\{x_i\},\{y_i\}$ be dual bases of $\g$ (i.e.\ $(x_h|y_k)=\delta_{hk}$). If $j>0$, let $(\cdot|\cdot)_j$ be normalized invariant form on $\g^0_j$ and set $\{x_i^j\},\{y_i^j\}$ to be dual bases of $\g^0_j$ with respect to $(\cdot|\cdot)_j$. Let $h^\vee_j$ be the dual Coxeter number of $\g^0_j$. For $\g^0_0$, let $\{x_i^0\},\{y_i^0\}$ be dual bases of $\g^0_0$ with respect to  $(\cdot|\cdot)_0=(\cdot|\cdot)_{|\g^0_0\times\g^0_0}$ and set $h_0^\vee=0$.\par
 If $\mathbf{k}=(k_0,\ldots,k_s)$ is a multi-index of levels we set
 $$
 V^{\mathbf{k}}(\g^0)= V^{k_0}(\g^0_0)\otimes\dots\otimes V^{k_s}(\g^0_s),
 $$
 and, assuming $k_j+h_j^\vee\ne 0$ for all $j$, we let
$$
V_{\mathbf{k}}(\g^0)= V_{k_0}(\g^0_0)\otimes\dots\otimes V_{k_s}(\g^0_s).
$$
Here $V^k(\g^0_0)$ denotes the corresponding Heisenberg vertex algebra. We also set $V_{\g^0}(\mu)=\otimes V_{\g^0_j}(\mu^j)$ and $L_{\g^0}(\mu)=\otimes L_{\g^0_j}(\mu^j)$. 

 If $\g^0_0=\C\varpi$ and $k\ne0$, then, setting $c=\frac{\varpi}{\sqrt{k(\varpi|\varpi)}}$, $V^k(\g^0_0)$ is the vertex algebra $M_c(1)$ generated by $c$ with $\lambda$--product $[c_\lambda c]=\l\vac$. We denote by  $M_c(1, r)$ the irreducible $M_c(1)$--module generated by the highest weight vector $v_r$ such that
$$ c(0) v_r = r v_r, \quad c(s) v_r = 0 \quad (s \ge 1). $$
In particular $L_{\g^0_0}(\mu^0)=M_c(1,\frac{\mu(\varpi)}{\sqrt{k(\varpi|\varpi)}})$.

 We consider $V^{k}(\g)$,  $V^{\mathbf{k}}(\g^0)$ and all their quotients, including $V_{k}(\g),V_{\mathbf{k}}(\g^0)$, as  conformal vertex algebras with conformal vectors $\omega_\g, \omega_{\g^0}$ given by the Sugawara construction:
$$
\omega_\g=\frac{1}{2(k+h^\vee)}\sum_{i=1}^{\dim \g} :y_i x_i :,\quad
\omega_{\g^0}=\sum_{j=0}^s\frac{1}{2(k_j+h_j^\vee)}\sum_{i=1}^{\dim \g^0_j} :y^j_i x_i^j :.
$$

Recall that, if a vertex algebra $V$ admits a conformal vector $\omega$ and the corresponding field is
$Y(\omega,z)=\sum_{n\in\ganz}\omega_nz^{-n-2}$, then, by definition of conformal vector, $\omega_0$ acts semisimply on $V$. If $x$ is an eigenvector for $\omega_0$, then the corresponding eigenvalue $\D_x$ is called the conformal weight of $x$. 

Let $V$ be a vertex algebra. Denote by $T$ the translation operator on $V$ defined by $T u = u(-2) {\bf 1}$
. If $U$, $W$ are subspaces in a vertex algebra then we define their {\it dot product}:
\begin{equation}\label{dot}U\cdot W=span(u(n)w\mid u\in U,\ w\in W,\  n\in\ZZ).
\end{equation}
The dot product is associative
and, if the subspaces are $T$-stable, commutative (cf. \cite{BK}).
The dot product in a simple vertex algebra does not have zero divisors: if $U\cdot V=\{0\}$ then either $U=\{0\}$ or $W=\{0\}$.

We let $\mathcal V_k(\g^0)$ denote the vertex subalgebra of $V_k(\g)$ generated by $x(-1)\vac$, $x\in\g^0$.  Note that, given $k\in\C$,  there is a uniquely determined multi-index $\mathbf{u}(k)$ such that $\mathcal V_k(\g_{\bar 0})$ is a quotient of $V^{\mathbf{u}(k)}(\g^0)$ hence, if $u_j(k)+h^\vee_j\ne 0$ for each $j$,  $\omega_{\g^0}$ is a conformal vector in $\mathcal V_k(\g^0)$. We will say that $\mathcal V_k(\g^0)$ is conformally embedded in $V_{k}(\g)$ if 
$
\omega_\g=\omega_{\g^0}.
$

Our aim is the study of conformal embeddings of $\mathcal V_k(\g^0)$ in $V_{k}(\g)$; in particular we will describe the classification of all conformal embeddings of $\mathcal V_k(\g_{\bar 0})$ in $V_{k}(\g)$.
The basis of our investigation is the following result, which is a variation of \cite[Theorem 1]{A}. Let $\g^1$ be the orthocomplement
of $\g^0$ in $\g$.  
\begin{theorem} \label{thm-intro}
In the above setting, $\mathcal V_k(\g^0)$
is conformally embedded in
$V_k(\g)$ if and only for any $x\in\g^1$ we have
\bea (\omega_{\g^0})_0x(-1)\vac = x(-1)\vac.
\label{eqn-intro}\eea
\end{theorem}

Assume that  $\g^1$ is completely reducible as a $\g^0$-module, and let  $$\g^1=\bigoplus_{i=1}^t V_{\g^0}(\mu_i)$$ be its decomposition.  Set $\mu_0=0$.

\begin{cor}\label{numcheckk}$ \mathcal{V}_{k}(\g^0)$ is
conformally embedded in $V_k({\g})$ if and only if
\bea\label{numcheck} && \sum_{j=0}^s\frac{ ( \mu_{i}^j , \mu _{i}^j + 2 \rho^j ) _j}{ 2 (u_j(k) + h_j
^{\vee} )} = 1  \eea
for all $i>0$.
\end{cor}

Assume $\g^0_0=\{0\}$ and that $\g^0$ is the set of fixed points an automorphism $\sigma$ of $\g$ of order $s$ and let $\g=\oplus_{i\in \ZZ/s\ZZ}\g^{(i)}$ be the corresponding eigenvalue decomposition.
Note that $\g^0=\g^{(0)}$ and that $\g^1=\sum_{i\ne 0} \g^{(i)}$. 
Since $\g^1$ is assumed to be completely reducible as $\g^0$-module, we have 
$$\g^{(i)}=\sum_{r\in I(i)}V(\mu_r),$$
where  $I(i)$ is a subsets of $\{1,\ldots t\}$.
The map $\sigma$ can be extended
to a finite order automorphism of the simple vertex  algebra
$V_{k} (\g)$ which induces the eigenspace decomposition
$$V_{k} (\g) =\oplus_{i\in \ZZ/s\ZZ}V_{k} (\g) ^{(i)} .$$
Clearly $V_{k} (\g) ^{(i)}$ are  $\widehat\g^0$--modules. 
Note that $\g_{\bar 0}$ is the fixed point set of the involution defined  $\sigma(x)=(-1)^ix$ for $x\in \g_{\bar i}$, so the above setting applies to $\g_{\bar 0}$.

The following result is a super analog of \cite[Theorem 3]{A}. For the sake of completeness, we provide a proof.

\begin{theorem} \label{general}
Assume that, if $\nu$ is the weight of a $\g^0$-primitive vector occurring in $V(\mu_i) \otimes V(\mu_j) $, then there is a  $\widehat\g^0$-primitive vector in $V_k(\g)$ of weight $\nu$ if and only if 
$\nu=\mu_r$ for some $r$.

Then $\mathcal V_k(\g^0)$ is simple and 
\begin{equation}\label{d}V_k(\g) = V_{\bf k}(\g^0)\oplus (\oplus_{i=1}^tL_{\g^0 }( \mu_i)). \end{equation}
\end{theorem}
\begin{proof}
Set $U=\C\vac\oplus \g^1\subset V_k(\g)$ and $\mathcal U=\mathcal V_k(\g^0)\cdot U$. It is enough to show that $V_k(\g)=\mathcal U$. Since $\mathcal U$ generates $V_k(\g)$ it suffices to check that $\mathcal U$ is a vertex subalgebra, which is equivalent to checking that $\mathcal U\cdot\mathcal U\subset \mathcal U$.

Since $\mathcal U\cdot \mathcal  V_k(\g^0)=\mathcal V_k(\g^0)\cdot \mathcal U$, we have
$$
\mathcal U\cdot \mathcal U=\mathcal V_k(\g^0)\cdot U\cdot \mathcal V_k(\g^0)\cdot U=\mathcal V_k(\g^0)\cdot U\cdot U
$$
so it is enough to check that $U\cdot U\subset\mathcal U$. Assume the contrary. Then there is $n$ such that $U(n)U+\mathcal U$ is nonzero in $V_k(\g)/\mathcal U$. Since $U$ is finite dimensional, we can assume $n$ to be maximal. It follows that there are $i,j$ and a vector $v$ in $V_{\g^0}(\mu_i)(n) V_{\g^0}(\mu_j)$ such that $(v+\mathcal U)/\mathcal U$ is nonzero. Since $V_{\g^0}(\mu_i)(n) V_{\g^0}(\mu_j)$ is finite--dimensional, it is $\g^0$--generated by $\g^0$--primitive vectors, thus we can assume that $v$ is $\g^0$--primitive. Let $V$, $W$ be $\g^0$--submodules of $V_{\g^0}(\mu_i)(n) V_{\g^0}(\mu_j)$ such that $v+W$ is the highest weight vector of $V/W\simeq V_{\g^0}(\nu)$. In particular, if $\eta$ is a weight occurring in $W$, then $\eta<\nu$. 

Note that, if $x\in \g^0$ and $m>0$, then 
\begin{equation}\label{VdotV}x(m)V_{\g^0}(\mu_i)(n) V_{\g^0}(\mu_j)\subset ad(x)(V_{\g^0}(\mu_i))(n+m)V_{\g^0}(\mu_j).
\end{equation}
In particular, $x(m)W\subset\mathcal U$. Set $\mathcal W=\mathcal V_k(\g^0)\cdot W$. By the above observation if $\eta$ is a weight occurring in $(\mathcal W+\mathcal U)/\mathcal U$, then $\eta<\nu$. This implies that $v+\mathcal U\notin (\mathcal W+\mathcal U)/\mathcal U$. If $\a$ is a positive root in $\g^0$, then $x_\a(0)v\in W$, so $x_\a(0)(v+\mathcal U)\in \mathcal W+\mathcal U$. Moreover, by \eqref{VdotV} again, $x(m)v\in\mathcal U$ for $m>0$. Set $\mathcal V=\mathcal V_k(\g^0)\cdot V$. It follows that $v+\mathcal U$ is $\widehat \g^0$--singular vector in $(\mathcal V+\mathcal U)/(\mathcal W+\mathcal U)$.

By our hypothesis, $\nu=\mu_r$ for some $r$. Then, by \eqref{numcheck}, $\Delta_v=1$ or $\Delta_v=0$. This implies that $v\in\g(-1)\vac+\C\vac$, thus $v\in \mathcal U$, and we reach a contradiction.
\end{proof}
 \begin{rem}\label{conditionsfinite}
The hypothesis of the previous theorem hold whenever for all primitive vectors of weight $\nu$ occurring in $V_{\g^0}(\mu_i)\otimes  V_{\g^0}(\mu_j) $, one has that either $\nu=\mu_r$ for some $r$ or 
\begin{equation}\label{finiote}
 \sum_{r=1}^s\frac{ ( \nu^r, \nu ^r+ 2 \rho^{r})_r }{ 2 (u_r(k)+ h^\vee _{r})}\not\in \ganz_+.
\end{equation}
 \end{rem}
Let now assume that   $\g^0_0=\C\varpi$ and that $\g^1$ decomposes as 
$$\g^1=V_{\g^0}(\mu)\oplus V_{\g^0}(\mu^*).$$
Observe that this is the case when $\g^0=\g_{\bar 0}$ and $\g_{\bar 0}$ is not semisimple.

 By a suitable choice of $\varpi$ we can assume that $\varpi$  acts as the identity on $V_{\g_0}(\mu)$ and as minus the identity on its dual. Define $\epsilon\in (\g^0_0)^*$ by setting \begin{equation}\label{epsilon}\epsilon(\varpi)=1.\end{equation}

If $q\in\ganz$, let $V_k(\g)^{(q)}$ be the eigenspace for the action of $\varpi{(0)}$ on $V_k(\g)$ corresponding to the eigenvalue $q$. Let $\{0,\nu_1,\cdots,\nu_m\}$ be the set of weights of $\g^0$--primitive vectors occurring in $V_{\g^0}(\mu) \otimes V_{\g^0}(\mu^*)$.

The following result is a super analog of \cite[Theorem 2.4]{AKMPP}.

\begin{theorem}\label{generalhs}
Assume that  $k\ne 0$ and  that $V_k(\g)^{(0)} $ does not contain\break 
$\widehat\g^0$-primitive  vectors of weight $\nu_{r}$, where $r= 1,\ldots,m$. 
Then  
\item \begin{equation}\label{dhs}\mathcal V_k(\g^0)\cong V_\mathbf{k}(\g^0)= V_k(\g)^{(0)}\end{equation}
and  $V_k(\g)^{(q)}$ is a simple $V_\mathbf{k}(\g^0)$--module, so that
 $V_k(\g)$ is completely reducible as a $\widehat\g^0$-module.
Moreover
\begin{align*}V_k(\g)^{(q)}&=\underbrace{V_{\g^0}(\mu)\cdot V_{\g^0}(\mu)\cdot \ldots\cdot V_{\g^0}(\mu)}_{\text{$q$ times}}&\text{if $q>0$,}\\
V_k(\g)^{(q)}&=\underbrace{V_{\g^0}(\mu^*)\cdot V_{\g^0}(\mu^*)\cdot \ldots\cdot V_{\g^0}(\mu^*)}_{\text{$|q|$ times}}\!\!\!\!\!\!\!&\text{if $q<0$}.
\end{align*}

\end{theorem}
 \begin{rem}\label{conditionsfinitered}
The assumption of Theorem \ref{generalhs} holds whenever, for $r=1,\ldots,m$,
\bea\label{finioths}
 \sum_{j=0}^t\frac{ ( \nu_{r}^j , \nu _{r}^j + 2 \rho^j) _j}{ 2 (u_j(k)+ h^\vee _{j})}\not\in \ganz_+.
\eea
\end{rem}

\section{Conformal levels} 

\begin{defi}A level $k\in\C$ is said to be a conformal level for the emebedding $\g^0\subset \g$ if
\begin{enumerate}
\item $k+h^\vee\ne0$,
\item $u_j(k)+h^\vee_j\ne 0$ for all $j$,
\item $\mathcal V_k(\g^0)$ is conformally embedded in $V_{k}(\g)$.
\end{enumerate}
\end{defi}

 \begin{theorem} \label{cf} The conformal levels for the embeddings $\g_{\bar 0}\subset \g$ are as follows.
 
\begin{enumerate}
\item \label{slmn}If ${\mathfrak g} =sl(m|n) $, $m>n\ge 2, m\ne n+1$, the conformal levels are $k=1,-1,\frac{n-m}{2}$; \\
   If  $n =1$,  $m \ge 3$,  the conformal levels are $k=-1,\frac{1-m}{2}$;  \\
    If  $m = n+1$, $m \ge 3$, the conformal levels are $k=1,-  \frac{1}{2}$;  \\
    If $m=2, n=1$ the only conformal level is $k=-\frac{1}{2}$;
\item \label{pslmm} If ${\mathfrak g} =psl(m|m) $, the conformal levels are $k=1,-1$;
\item  \label{Bnm} If ${\mathfrak g}$ is of type $B(m,n)$, the    conformal levels are $k=1,\frac{3-2m+2n}{2}$   if $m\ne n$, $k=\frac{3}{2}$ if $m=n$, and $k=-\frac{2n+3}{2}$ if $m=0$.
\item  \label{Dnm} If ${\mathfrak g}$ is of type $D(m,n)$,  the conformal levels are $k=1,2-m+n$   if $m\ne n$ and $k=1$ if $m=n$;
\item  \label{Cn1} If ${\mathfrak g}$ is of type $C(n+1)$,  the conformal levels are $k=1,1+n$  if $n>1$ and $k=2$ if $n=1$;
\item  \label{F4} If ${\mathfrak g}$ is of type $F(4)$,  the conformal levels are $k=1,-\frac{3}{2}$;
\item  \label{G3} If ${\mathfrak g}$ is of type $G(3)$,  the conformal levels are $k=1,-\frac{4}{3}$;
\item  \label{D21a} If ${\mathfrak g}$ is of type $D(2,1,a)$,  the conformal levels are $k=1,-1-a,a$  for $a\notin{1,-1/2, -2}$;
 the only conformal level for $D(2,1; -\frac{1}{2})$ is  $k=\half$; the only conformal level for both $D(2,1;1)$ and $D(2,1;-2)$ is $k=1$.
\end{enumerate}
\end{theorem}

\begin{proof}
We apply Corollary \ref{numcheckk} and solve \eqref{numcheck}.  For each case we list here the relevant data. 

\noindent (1) $\g=sl(m|n)$, $m>n \ge 1$: in this case $\g_{\bar 0}= \C \varpi\times sl(m)\times sl(n)$ (disregard the rightmost factor when $n=1$), where 
$$
\varpi=\frac{1}{n-m}\begin{pmatrix} nI_m&0\\0&mI_n\end{pmatrix}.
$$
The form is the supertrace form, hence it restricts to the normalized invariant form on $sl(m)$ and to its opposite on $sl(n)$. It follows that $u_0(k)=u_1(k)=k$ and $u_2(k)=-k$. As $\g_{\bar 0}$-module,
$$
\g_{\bar 1}=V_{\C\varpi}(\epsilon)\otimes V_{sl(m)}(\omega_1)\otimes V_{sl(n)}(\omega_{n-1})\oplus V_{\C\varpi}(-\epsilon)\otimes V_{sl(m)}(\omega_{m-1})\otimes V_{sl(n)}(\omega_{1})
$$
Since
$(\epsilon,\epsilon)=\frac{n-m}{mn}$ and, in the normalized invariant form of  $sl(r)$, 
$$
(\omega_1,\omega_1+2\rho)=(\omega_{r-1},\omega_{r-1}+2\rho)=\frac{(r-1)(r+1)}{r},
$$
equation \eqref{numcheck} reads for both factors of $\g_{\bar 1}$
$$
\frac{(m-1)(m+1)}{2m(k + m)}+ \frac{(n-1)(n+1)}{2n(-k + n)} + \frac{n-m}{2mnk}=1
$$
whose solutions are $1,-1,\half(n-m)$  if $n>1$ and $-1,\half(1-m)$ if $n=1$. Next we have to check that the previous values are not critical for $\g$: this excludes $k=1$ when $m=n+1$.

\noindent (2) $\g=psl(m|m)$: in this case $\g_{\bar 0}= sl(m)\times sl(m)$.
The form is the form induced by the supertrace form on $sl(m|m)$, hence it restricts to the normalized invariant form on the first $sl(m)$-factor of $\g_{\bar 0}$ and to its opposite on the second factor. It follows that $u_1(k)=k$ and $u_2(k)=-k$. As $\g_{\bar 0}$-module,
$$
\g_{\bar 1}= V_{sl(m)}(\omega_1)\otimes V_{sl(m)}(\omega_{m-1})\oplus  V_{sl(m)}(\omega_{m-1})\otimes V_{sl(m)}(\omega_{1}),
$$
thus equation \eqref{numcheck} reads for both factors of $\g_{\bar 1}$
$$
\frac{(m-1)(m+1)}{2m(k + m)}+ \frac{(m-1)(m+1)}{2m(-k + m)} =1
$$
whose solutions are $1,-1$.

\noindent (3) $\g$ of type $B(m,n)$: in this case $\g_{\bar 0}= so(2m+1)\times sp(2n)$.
The form is half the supertrace form. If $m>1$, $(\cdot|\cdot)$ restricts to the normalized invariant form on $so(2m+1)$ and to $-1/2$ the normalized invariant form on  $sp(2n)$. It follows that $u_1(k)=k$ and $u_2(k)=-k/2$.
As $\g_{\bar 0}$-module,
$$
\g_{\bar 1}= V_{so(2m+1)}(\omega_1)\otimes V_{sp(2n)}(\omega_{1}),
$$
thus equation \eqref{numcheck} reads
$$
\frac{m}{k+2m-1}+\frac{2n+1}{2(-k+2n+2)}=1.
$$
Its  solutions are $1,\frac{3-2m+2n}{2}$. Next we have to check that the previous values are not critical for $\g$: this excludes $k=1$ when $m=n$ .

If $m=1$ then $(\cdot|\cdot)$ restricts to twice the normalized invariant form on $so(3)$ and to $-1/2$ the normalized invariant form on  $sp(2n)$. It follows that $u_1(k)=2k$ and $u_2(k)=-k/2$.
As $\g_{\bar 0}$-module,
$$
\g_{\bar 1}= V_{so(3)}(2\omega_1)\otimes V_{sp(2n)}(\omega_{1}),
$$
thus equation \eqref{numcheck} reads
$$
\frac{1}{k+1}+\frac{2n+1}{2(-k+2n+2)}=1.
$$
Its  solutions are $1,\frac{1+2n}{2}$. Next we have to check that the previous values are not critical for $\g$: this excludes $k=1$ when $n=1$ .

Finally, in the case $m=0$, $(\cdot|\cdot)$ restricts to  $1/2$ the normalized invariant form on  $sp(2n)$. It follows that  $u_1(k)=k/2$.
As $\g_{\bar 0}$-module,
$$
\g_{\bar 1}= V_{sp(2n)}(\omega_{1}),
$$
thus equation \eqref{numcheck} reads
$$
\frac{2n+1}{2(k+2n+2)}=1.
$$
whose unique solution is $-\frac{3+2n}{2}$. This value is never critical for $\g$.

\noindent (4) $\g$ of type $D(m,n)$: in this case $\g_{\bar 0}= so(2m)\times sp(2n)$.
The form is half the supertrace form, hence it restricts to the normalized invariant form on $so(2m)$ and to $-1/2$ the normalized invariant form on  $sp(2n)$. It follows that $u_1(k)=k$ and $u_2(k)=-k/2$.
As in case (3), as $\g_{\bar 0}$-module, 
$$
\g_{\bar 1}= V_{so(2m)}(\omega_1)\otimes V_{sp(2n)}(\omega_{1}),
$$
thus equation \eqref{numcheck} reads
$$
\frac{2m-1}{2(k+2m-2)}+\frac{n+1/2}{2(1-k/2+n)}=1.
$$
Its solutions are $1, 2-m+n$. Examining the critical values  excludes $k=2$ when $m=n$.

\noindent (5) $\g$ of type $C(n+1)$: in this case $\g_{\bar 0}= \C \varpi\times sp(2n)$, where 

$$
\varpi=\left(
\begin{array}{c|c}
H&0 \\
\hline
0 & 0
\end{array}
\right),\quad H=\begin{pmatrix} 1&0\\0&-1\\\end{pmatrix}.
$$
 The form is $1/2$ the supertrace form, and $u_0(k)=k$ and $u_1(k)=-(1/2)k$. As $\g_{\bar 0}$-module, 
$$
\g_{\bar 1}= V_{\C\varpi}(\epsilon)\otimes V_{sp(2n)}(\omega_{1})\oplus V_{\C\varpi}(-\epsilon)\otimes V_{sp(2n)}(\omega_{1}),
$$
thus equation \eqref{numcheck} reads
$$
\frac{1}{2k}+\frac{n+1/2}{2-k+2n}=1.
$$
Its solutions are $1,1+n$, and $k=1$ should be excluded when $n=1$.

\noindent (6) $\g$ of type $F(4)$: in this case $\g_{\bar 0}= sl(2)\times so(7)$. We choose the invariant form  in such a way  that it restricts to the normalized invariant form on $so(7)$, and $u_1(k)=-2/3k, 
u_2(k)=k$.
We have $$
\g_{\bar 1}= V_{sl(2)}(\omega_1)\otimes V_{so(7)}(\omega_{3})
$$
thus equation \eqref{numcheck} reads
$$
\frac{9}{8(-k+3)}+\frac{21}{8(k+5)}=1.
$$
Its solutions are $-\frac{3}{2}, 1$.

\noindent (7) $\g$ of type $G(3)$: in this case $\g_{\bar 0}= sl(2)\times G_2$. We choose the invariant form in such a way    that it restricts to the normalized invariant form on $G_2$, and $u_1(k)=-3/4k, 
u_2(k)=k$.
We have $$
\g_{\bar 1}= V_{sl(2)}(\omega_1)\otimes V_{G_2}(\omega_{1})
$$
thus equation \eqref{numcheck} reads
$$
\frac{3}{-3k+8}+\frac{2}{k+4}=1.
$$
Its solutions are $-\frac{4}{3}, 1$.

\noindent (8) $\g$ of type $D(2,1;a)$: in this case $\g_{\bar 0}= sl(2)\times sl(2)\times sl(2)$. We choose the invariant form in such a way   that it restricts to the normalized invariant form on the first $sl(2)$  and $u_1(k)=k,  u_2(k)=k/a, u_3(k)=-\frac{k}{1+a}$.
We have $$
\g_{\bar 1}= V_{sl(2)}(\omega_1)\otimes V_{sl(2)}(\omega_{1})\otimes V_{sl(2)}(\omega_{1})
$$
thus equation \eqref{numcheck} reads
$$
\frac{3}{4(k+2)}+\frac{3}{4(\frac{k}{a}+2)}+\frac{3}{4(-\frac{k}{1+a}+2)}=1.
$$
Its solutions are $1,-1-a,a$ and some cases are excluded as specified in the statement.
\end{proof}

\begin{remark} \label{D21}
 Note that   for $\g$ of type $D(2,1;a)$  one can choose the parameter $a$ so that the subalgebra $\mathcal V _k(\g_{\bar 0})$ is non simple. For example, for $a= -\tfrac{3}{4}$, one can show that  $\mathcal V_k(\g_{\bar 0}) =V_1(sl(2)) \otimes V^{-4/3} (sl(2)) \otimes V_{-4}(sl(2))$. These non simple embeddings will  be investigated in our future papers.
\end{remark}

The next result gives some examples of conformal embeddings for $\g^0\subset \g$ with $\g^0$ not a Lie algebra.

\begin{theorem}\label{super-super}\ 

\noindent (1) Assume $n\ne m,m-1$.The conformal levels for the embedding $gl(n|m)\subset sl(n+1|m)$ are $k=1$ and $k=-\frac{n+1-m}{2}$.

\noindent (2) The conformal levels for the embedding $sl(2) \times osp(3|2)\subset G(3)$ are $k=1$ and $k=-4/3$.
\end{theorem}
\begin{proof}
Consider first the embedding $gl(n|m)\subset sl(n+ 1|m)$. We have
 $$\g^0=\C\varpi\oplus sl(n|m),\qquad \varpi=\frac{1}{n-m+1}I_{m,n}$$
 where
 $$
I_{m,n}= \begin{pmatrix} m-n & 0 & 0 \\0 & I_n & 0 \\ 0 & 0 &I_m\end{pmatrix}.$$
Then  $\g=\g^0\oplus \g^1$, where 
$$\g^1=V_{\C\varpi}(\epsilon)\otimes \C^{n|m}\oplus V_{\C\varpi}(-\epsilon)\otimes (\C^{n|m})^*.$$

The invariant form is the supertrace form.  We now compute the conformal levels. Equation \eqref{numcheckk} becomes in the present case
$$\frac{n-m+1}{2k(n-m)}+\frac{(m-n+1)(1-m+n)}{2(k+n-m)(m-n)}=1.$$
 Its solutions are $k=1$ and $k=-\frac{n+1-m}{2}$.
 
 Consider now the case of $sl(2) \times osp(3|2)\subset G(3)$. 
Recall that $\Dp$ is the distinguished positive set of roots for $G(3)$. Let $\a_1,\a_2,\a_3$ be the corresponding  simple roots ordered as in Table 6.1 of \cite{Kw}. Let $\omega_i^\vee\in\h$ be such that $\a_j(\omega_i^\vee)=\delta_{ji}$. Then $\g^0$ is the fixed point set of $\sigma=e^{\pi\sqrt{-1}\omega^\vee_2}$. In particular, one sees that 
$$\g^0=\h\oplus\bigoplus\limits_{\a(\omega_2^\vee)\text{ even}}\g_\a,\quad\g^1=\bigoplus\limits_{\a(\omega_2^\vee)\text{ odd}}\g_\a.
$$ 

From this explicit description one sees that one can choose the simple roots for $osp(3|2)$ to be $\be_1=\a_1$ and $\be_2=2\a_2+\a_3$ and that 
 $$
 \g^1= V_{sl(2)}(\omega_1)\otimes V_{osp(3|2)}(\be_1+3/2\beta_2).
 $$
 Equation \eqref{numcheckk} becomes in the present case
$$\frac{3}{4(k+2)}+\frac{3}{8(\tfrac{3}{2}k-1)}=1$$
whose solutions are $k=1$ and $k=-4/3$.
\end{proof}
\begin{remark}
Note that a  conformal level is either $1$ or collapsing (see \cite{AKMPP-JA} for the notion of collapsing level). There are however a few negative collapsing levels which are not conformal.
\end{remark}

\section{Decompositions for the embedding $\g_{\bar 0}\subset \g$}

\subsection{Easy cases }\label{easy}

In the following proposition we list the cases when \eqref{d},                 \eqref{dhs} hold since conditions \eqref{finiote}, \eqref{finioths} are verified. To simplify some formulas we also introduce the following notation for some $V_{-1}(sl(m))$--modules:
$$   U_ s ^{(m)}  = L_{sl(m)} ( s \omega_1), \quad   U_ {-s}  ^{(m)}  = L_{sl(m)} (s \omega_{m-1} ),\ s\in\Z_+.$$

\begin{proposition}\label{justapply}
\bea\nonumber
(1)\ V_{-4/3}(G(3))&=&V_1(sl(2))\otimes V_{-4/3}(G_2) \oplus L_{sl(2)}(\omega_1)\otimes L_{G_2}(\omega_1),\\\nonumber
(2)\ V_{-3/2}(F(4))&=&V_1(sl(2))\otimes V_{-3/2}(so(7)) \oplus L_{sl(2)}(\omega_1)\otimes L_{so(7)}(\omega_3),\\\nonumber
(3)\ V_{1}(B(m,n))&=&V_1(so(2m+1))\otimes V_{-1/2}(sp(2n))\\\nonumber
&& \oplus L_{so(2m+1)}(\omega_1)\otimes L_{sp(2n)}(\omega_1), m\ne n,\\\nonumber
(4)\ \ V_{k}(B(0,n))&=& V_{-(2n+3)/4}(sp(2n))\\\nonumber
&& \oplus  L_{sp(2n)}(\omega_1),\ k=-(2n+3)/2,\\\nonumber
(5)\ V_{1}(D(m,n))&=&V_1(so(2m))\otimes V_{-1/2}(sp(2n)) \\\nonumber
&&\oplus L_{so(2m)}(\omega_1)\otimes L_{sp(2n)}(\omega_1),\,\, m\ne n+1,\\\nonumber
(6)\ V_{1}(C(n+1))&=& M_c(1)\otimes V_{-1/2}(sp(2n))\\\nonumber
&&\oplus\sum_{q\in {\Z\backslash\{0\}} } M_c(1,2q)\otimes V_{-1/2}(sp(2n))\\\nonumber
&&\oplus\sum_{q \in {\Z} } M_c(1,2q+1)\otimes L_{sp(2n)}(\omega_1),\\\nonumber
(7)\ V_{1}(sl(m|n))&=&M_c(1)\otimes V_1(sl(m))\otimes V_{-1}(sl(n))\\\nonumber
&&\oplus \sum_{q \in\Z\backslash\{0\}}  M_c(1,\sqrt{\tfrac{n-m}{nm}}qm)\otimes V_1(sl(m)) \otimes U^{(n)}_{-qm} \\\nonumber
&&\oplus \!\!\!\!\!\!\!\!\sum_{
\substack{j =1,\ldots,m-1\\q\in\mathbb Z}}\!\!\!\!\!\!\!\!\!M_c(1,\sqrt{\tfrac{n-m}{nm}}(qm+j))\otimes L_{sl(m)} (\omega_j)  \otimes U_{ -qm - j} ^{(n)}.\nonumber
\eea
In case (7), $m\ne n, n-2, m \ge 2, n \ge 3$.
\end{proposition}
\begin{proof}In cases (1)--(5) one needs only to check \eqref{finiote}. As an example, here we give  the details only for case (4): $\g=osp(1|2n)$ and the invariant form is $(x|y)=-\tfrac{1}{2}str(xy)$. Moreover, as $\g_{\bar 0}$-module,
$
\g_{\bar 1}= V_{sp(2n)}(\omega_{1}),
$
thus 
$$\g_{\bar 1}\otimes \g_{\bar 1}=\begin{cases}V_{sp(2n)}(2\omega_1)\oplus V_{sp(2n)}(\omega_2)\oplus V_{sp(2n)}(0)&\text{if $n>1$}\\
V_{sp(2n)}(2\omega_1)\oplus V_{sp(2n)}(0)&\text{if $n=1$}
\end{cases}.
$$
Since, $u_1(k)=1/2$ and $h^\vee _{1}=n+1$,
$$
\frac{ ( 2\omega_1,2\omega_1+ 2 \rho_0)_1}{ 2 (u_1(k)+ h^\vee _{1})}=\frac{2 n+2}{2 \left(\frac{1}{4} (-2 n-3)+ n+1\right)}=\frac{2}{2 n+1}+2
$$
and, if $n>1$,
$$
\frac{ ( \omega_2,\omega_2+ 2 \rho_0)_1}{ 2 (u_1(k)+ h^\vee _{1})}=\frac{2n}{2 \left(\frac{1}{4} (-2 n-3)+ n+1\right)}=-\frac{2}{2 n+1}+2
$$
thus \eqref{finiote} holds.

In cases (6), (7) the only nontrivial step is the computation of the decomposition for $V_1(C(n+1))$ and $V_1(sl(n|m))$.

The decomposition for $V_1(C(n+1))$ follows readily from Theorem \ref{generalhs}, the fusion rule for $V_{-1/2}(sp(2m))$ (see  \cite{We})
\bea\label{fusionCm}
L_{sp(2m)} (\omega_1) \times L_{sp(2m)} (\omega_1) = V_{-1/2} (sp(2m)),
\eea
and the well known fusion rules
\bea M_c( 1, r) \times M_c(1,s) = M_c(1, r+s). \label{fr-heisenberg} \eea

For $V_{1}(sl(n|m))$, consider the $V_1(sl(m))$--modules
$$Z^{(m)}_{0}=V_1(sl(m)),\ Z^{(m)}_{j}=L_{sl(m)}(\omega_j),\ j=1,\ldots,m-1.
$$
 and recall the following fusion rules:
\bea  U_ {s_1} ^{(n)} \times  U_ {s_2} ^{(n)}  &  = & U_{s_1+ s_2} ^{(n)} \quad (s_1, s_2 \in {\Z}), \label{prvi} \\
 Z_ {j_1} ^{(m)} \times  Z_ {j_2} ^{(m)}  &=& Z ^{(m)}  _{j_1+ j_2 \mod m   }\quad (  j_1, j_2 \in \{0, \dots , m-1\} ) \label{second}. \eea

 The fusion rules (\ref{prvi}) were proved in \cite{AP2} and (\ref{second}) in \cite{DL}.

Since in this case $V_{\g_{\bar 0}}(\mu)=L_{\C\varpi}(\epsilon)\otimes  Z^{(m)}_1\otimes U^{(n)}_1$, $V_{\g_{\bar 0}}(\mu^*)=L_{\C\varpi}(-\epsilon)\otimes  Z^{(m)}_{m-1}\otimes U^{(n)}_{-1}$ and $L_{\C\varpi}(\pm\epsilon)=M_c(1,\pm\sqrt{\tfrac{n-m}{nm}})$, we obtain from \eqref{fr-heisenberg}, \eqref{prvi}, \eqref{second} that
$$\underbrace{V_{\g_{\bar 0}}(\mu)\cdot V_{\g_{\bar 0}}(\mu)\cdot \ldots\cdot V_{\g_{\bar 0}}(\mu)}_{\text{$q$ times}}=M_c(1,\sqrt{\tfrac{n-m}{nm}}q)\otimes Z^{(m)}_{q \mod m   }\otimes U^{(n)}_{-q}$$
and
$$\underbrace{V_{\g_{\bar 0}}(\mu^*)\cdot V_{\g_{\bar 0}}(\mu^*)\cdot \ldots\cdot V_{\g_{\bar 0}}(\mu^*)}_{\text{$-q$ times}}=M_c(1,\sqrt{\tfrac{n-m}{nm}}q)\otimes Z^{(m)}_{q \mod m   }\otimes U^{(n)}_{-q},$$
so Theorem \ref{generalhs} provides the desired decomposition.
\end{proof}

\begin{remark}
In  Subsection  \ref{CCn1} below we derive the decomposition above for $V_1(C(n+1))$ using  a different  approach that has the advantage of clarifying the vertex algebra structure of the even part of $V_1(C(n+1))$.
\end{remark}

\subsection{Another approach to the case  ${\mathfrak g}=sl(m|n)$, $k = 1$  }\label{slmn-realization}
Here we give a different approach to the decomposition of $V_1(sl(m|n))$ as $\mathcal V_1(\g_{\bar 0})$--module that extends the result in Proposition \ref{justapply} to the missing $m=n-2$ case.
\begin{theorem}\label{slmn-realisation}  Let $\g=sl(m|n)$ with $m\ne n$, $m \ge 2$, $n \ge 3$. Then 
\begin{equation}\label{slmn-simple}
\mathcal V_1(\g_{\bar 0})=V_1(sl(m))\otimes V_{-1}(sl(n))\otimes M_c(1)
\end{equation}
and the decomposition in Proposition \ref{justapply} (6) holds.
In particular, $V_1 ( {\mathfrak g} )$ is a simple current extension of the vertex algebra $V_{1} (sl(m)) \otimes V_{-1} (sl(n)) \otimes M_c(1) $
\end{theorem}
\begin{proof}It is enough to prove that  the action of $\mathcal V_1(\g_{\bar 0})$ on $V_1(sl(m|n))$ is semisimple. In fact, in such a case, by the fusion rules 
(\ref{prvi}), (\ref{second})   and  (\ref{fr-heisenberg}), Theorem \ref{generalhs} can be applied.
The semisimplicity follows from the free field realization of $V_1(gl(m|n))$ in   $M_{(2m, 2n)} = F_{(m)} \otimes M_{(n)}$ where
$F_{(m)}$ and $M_{(n)}$ are respectively the fermionic and Weyl vertex algebras (cf.  Section \ref{FFR}). In fact the composition of the  embeddings
$$\mathcal V_1(gl(m)\times gl(n))\subset V_1(gl(m|n))\subset F_{(m) }\otimes M_{(n)}$$
is the tensor product of the embeddings of $\mathcal V_1(gl(m))$ in $F_{( m) } $ and of $\mathcal V_{-1}(gl(n))$ in $M_{(n)}$. It is well known that  $F_{(m)}$ is completely reducible as $\widehat{gl(m)}$-module. The fact that $M_{(n)}$ is completely reducible as $\widehat{gl(n)}$-module is proven in \cite{AP2}.
\end{proof}

\subsection{The case $ {\mathfrak g} =psl(m|m) $,   $k=1$.  }
\label{psl(n,n)-section}

The approach of \S\ \ref{slmn-realization} readily extends to the case of $\g=psl(m|m)$:
\begin{theorem} \label{psl(n,n)-dec} Assume that $m \ge 3$.  Then
\begin{equation}\label{decpslnn}
V_{1}(psl(m|m))= \sum_{j = 0} ^{m-1} \sum_{q \in {\Z}} L_{sl(m)} (\omega_j)  \otimes U_{ -q  m - j} ^{(m)}. 
\end{equation}
In particular, $V_1 ( {\mathfrak g} )$ is a simple current extension of the vertex algebra \break $V_1 (sl(m)) \otimes V_{-1} (sl(m))$.
\end{theorem}
\begin{proof}The proof follows as in Theorem \ref{slmn-realisation} from the semisimplicity of the action of $\mathcal V_1(sl(m)\times sl(m))$ on $V_1(psl(m|m))$. To prove semisimplicity, let $I\in sl(m|m)$ be the identity matrix. Then we have $\mathcal V_1(sl(m|m))$--modules
$$\mathcal V_1(sl(m|m))\cdot I\subset \mathcal V_1(sl(m|m))\subset V_1(gl(m|m)).
$$
 Moreover the map $x\mod \C I\mapsto x(-1)\vac\mod \mathcal V_1(sl(m|m)\cdot I$ extends to a vertex algebra map from $V^1(psl(m|m))$ to $\mathcal V_1(sl(m|m))/(\mathcal V_1(sl(m|m)\cdot I)$. Let $\mathcal V_1(psl(m|m))$ be the image of this map. Since $\mathcal V_1(sl(m)\times sl(m))$ acts semisimply on $V_1(gl(m|m))$, it acts semisimply also on $\mathcal V_1(psl(m|m))$ and therefore on its quotient $V_1(psl(m|m))$. 
\end{proof}
\begin{remark}
The simple current extension in the theorem above is a  super analog of extensions studied in \cite{KMPX}. We should also mention that the super-character formula  for  $V_1 ( {\mathfrak g} )$ is presented in \cite{AdM-2018}.
\end{remark}

The case $m=2$ was given by Creutzig and Gaiotto as one of the main results in their paper \cite{CG}. We shall here only state their result on the decomposition.
\begin{proposition} \cite[Remark 9.11]{CG} Assume that $m=2$. We have:
$$V_1(\g) =   \bigoplus_ {i = 0} ^{\infty}  \left(     (2i+1 )  Z_0 ^{(2)}  \otimes  U _ { 2i} ^{(2)} \right)\oplus \bigoplus_ {i = 0} ^{\infty}     \left(  (2i+2)  Z_1 ^{(2)}  \otimes  U _ { 2i + 1} ^{(2)} \right).  $$ 
In particular,  $V_1(\g)$ is an extension of  $V_1 (sl(2)) \otimes V_{-1} (sl(2))$ which is not  simple current.
\end{proposition}

\subsection{The case  ${\mathfrak g}=sl(m|n)$, $k = -h^{\vee} / 2$  }

An interesting case is to consider embeddings to $V_k(\g)$ where   ${\mathfrak g}=sl(m|n)$ and conformal level is $k = -h^{\vee} / 2$. In this paper we only consider the case $n=1$. The general case is more complicated and we plan to consider it in our future work.
\subsubsection{The case $n=1$} 
 
In this case, $k =  1/ 2 - m$. This level is   admissible for $sl(2 m)$, and the fusion rules were determined in  \cite[Proposition 5.1]{AKMPP-JJM}. We have:

\begin{itemize}
\item The set of irreducible $V_k(sl(2m))$--modules in $KL_k$ is
$$ \{ L_{sl(2m)} (\omega_i ) \ \vert \ i = 0, \dots, 2m -1 \}. $$
\item  The following fusion rules hold:
$$ L_{sl(2m)} (\omega_{i_1}) \times   L_{sl(2m)} (\omega_{i_2}) =  L_{sl(2m)} ( \omega_{i_3}) $$
where $ 0 \le i_1, i_2 , i_3 \le 2m -1$ such that $i_1 + i_2 \equiv i_3  \ \mbox{mod}  (2m)$. 
\end{itemize}

 Now we are ready to analyse the conformal embedding $sl(2 m) \times {\C} \hookrightarrow sl (2m, 1)$ at level $k$. We have:
\begin{itemize}
\item $V_k (\g) ^0  \cong V_k( sl(2m)) \otimes M_c(1)$, where
$c= (I_{2m,1}) _{(-1)}{\bf 1}$.
\item  $V_k (\g) ^j $ is irreducible $V_k( sl(2m)) \otimes M_c(1)$ on which $c(0)$ acts as  $j (2m-1). $ In particular,
$$   V_k (\g) ^j \cong L_{sl(2m)} (\omega_{i_j } ) \otimes M_c(1,  j (2m -1)).  $$
where $ 0 \le i_j  \le 2m -1$ such that $i_j   \equiv  j \ \mbox{mod}  (2m)$.
\item Now  we get:
$$ \mbox{Com} (V_k (sl(2m)), V_k(\g)) = F_{2m (2m-1)} =  \bigoplus_{i \in {\mathbb Z}} M_c( 1, i 2m (2m-1)) $$
where $ F_{2m (2m-1)}$ is the rank one lattice vertex algebra $V_{\Z \alpha}$ such that $\langle \alpha, \alpha \rangle = 2 m (2m-1). $
\end{itemize}

\subsection {The case  ${\mathfrak g}$ of type $D(m,n)$, $k = 1$  }\ 
The following approach to the decomposition includes  also the case 
$m=n+1$, not covered by Theorem \ref{justapply}.

Assume first that  $n \ge 2$. We consider  the universal affine vertex algebra $V^{1} (\g)$. 
The vector
\begin{equation}\label{Omega} \Omega =  \left( X_{\varepsilon_1 + \varepsilon_2} (-1) ^2 - X_{2 \varepsilon_1} (-1)  X_{2 \varepsilon_2} (-1)  \right){\bf 1}
\end{equation}
 is a singular vector in $V^{-1/2} (sp(2n) )$, and  it defines a non-trivial graded ideal  $J^{1} (\g)= V^{1} (\g) \cdot \Omega $ in $V^1 (\g)$. Set
 $$ \mathcal Q^{1 } (\g) = V^{1} (\g) /  J ^{1} (\g). $$
 
 \begin{proposition} \label{dec-dmn} Assume that $n \ge 2$.
 \item[(1)] The even subalgebra of $\mathcal Q^{1} (\g)$  is isomorphic  to $$V_{-1/2} (sp(2n) ) \otimes V_1 (so(2m) ), $$
 \item[(2)] The following decomposition holds
 $$ \mathcal Q^{1} (\g) = V_{-1/2} (sp(2n)) \otimes V_1 (so(2m) ) \oplus L_{sp(2n)} (\omega_1) \otimes L_{so(2m)}  (\omega_1). $$
 \item[(3)] $\mathcal Q^{1} (\g) =  V_{1} (\g)$. 
 \end{proposition}
 \begin{proof}
First we notice the following facts
 \begin{itemize}
 \item The maximal ideal in $V^ {-1/2} (sp(2n))$ is generated by $\Omega$ (cf. \cite{A-1994}).
 \item The maximal ideal in $V^ {1} (so(2m) )$ is generated by $X_{\bar \Theta}(-1) ^2 {\bf 1}$, where $\bar \Theta$ is the highest root in $so(2m)$, and $X_{\bar \Theta} $ is a corresponding root vector.
 \item $X_{\bar \Theta}(-1) ^2 {\bf 1} \in V^{1} (\g)\cdot \Omega .$
 \end{itemize}
 This implies that $\mathcal Q^{1} (\g) $ contains a vertex subalgebra $U$  isomorphic to $$V_{-1/2} (sp(2n)) \otimes V_1 (so(2m) ). $$
 By using the decomposition of $\g$ as $sp(2n) \times so(2m)$--module, the semi--simplicity of $U$--modules, we find a $U$--submodule $M$ inside of $\mathcal V^{1} (\g) $ which is isomorphic to
 $$L_{sp(2n)} (\omega_1) \otimes L_{so(2m)}  (\omega_1).$$ Recall the following the fusion rules
 \begin{itemize}
\item \cite{We}
$L_{sp(2n)} ( \omega_1) \times L_{sp(2n)} (\omega_1) = V_{-1/2} (sp(2n))$.
\item \cite{DL} $L_{so(2m)}  (\omega_1) \times L_{so(2m)}  (\omega_1) = V_1 (so(2m))$.
\end{itemize}
which implies the fusion rules $M \times M = U$  and therefore 
 $U \oplus M$ is a vertex subalgebra of $\mathcal Q^{1} (\g)$. Since $U \oplus M$ contains all generators of  $\mathcal Q^{1} (\g)$, we get the assertion.
 \end{proof}
 
 The above decomposition holds also in the case $n=1$.
 \begin{proposition} \label{dec-dm1}Assume that $n=1$. Then we have:
 $$  V_{1} (\g) = V_{-1/2} (sl(2)) \otimes V_1 (so(2m) ) \oplus L_{sl(2)} ( \omega_1) \otimes L_{so(2m)}  (\omega_1). $$
 \end{proposition}
 \begin{proof}
 From the  explicit realization we conclude that $V_{1} (\g)$ has a subalgebra isomorphic to  $V_{-1/2} (sl(2)) \otimes V_1 (so(2m) )$ and contains the  $V_{-1/2} (sl(2)) \otimes V_1 (so(2m) )$--module $ L_{sl(2)} (\omega_1) \otimes L_{so(2m)}  (\omega_1)$. The claim follows by using fusion rules.
 \end{proof}

 \begin{remark}
 The same argument applied to $\g$ of type $B(n,m)$ yields the same result of Theorem \ref{justapply} and, moreover, shows that, if $n\ge 2$, the vector $\Omega$ given in \eqref{Omega} generates the maximal ideal in $V^1 (\g)$.
 \end{remark}

\subsection {The case  ${\mathfrak g}$ of type $D(m,n)$, $k = 2-m+n$. }

 We conjecture that  in this case $\mathcal V_{-2} (\g_{\bar 0})$ will be a simple vertex algebra, and that $V_k(\mathfrak g)$ is the semi-simple $\mathcal V_{-2} (\g_{\bar 0})$--module. But at the moment  we can prove these conjectures only in the case 
$\g = osp(2n + 8|2n)$ and conformal level $k=-2$. Note that $k=-2$ is also a collapsing level, and therefore  we can use results from \cite{AKMPP-IMRN}.
The general case, i.e., when $k$ is non-collapsing, is at the moment beyond the range of applicability of our methods.
 
\begin{theorem}  \
(1). The vertex algebra $\mathcal V_{-2} (\g_{\bar 0})$ is simple and it is isomorphic to $V_{-2} (so(2n +8)) \otimes V_1 (sp(2n))$.
 
\noindent (2).
We have the following decomposition
     	$$ V_{-2}(\g) = \bigoplus_{i = 0} ^{n} L_{-2} (i\omega_1) \otimes L_1(\omega_i). $$
\end{theorem}
The proof will be given in Section  \ref{osp-conformal}. It uses explicit realization of one non simple quotient of $V^{-2}(\g)$, the fusion rules and the representation theory of the vertex algebra  $V_{-2} (so(2n +8))$ from \cite{AKMPP-IMRN}.
  
 \subsection{The case $\g = spo(2|3)$, $k =-3/4$  }
 We now discuss the case of $\g$ of type $B(1,1)$. According to Theorem \ref{cf}, the only conformal level is $k=3/2$. In this case
  $\g_{\bar 0} =sp(2)\times so(3)\simeq sl(2) \times sl(2)$. Let $\a,\be$ be roots of $sp(2)$ and $so(3)$ respectively. In the normalization of the form $(\cdot|\cdot)$ used in  Theorem \ref{cf} we have $(\a|\a)=-4$ and $(\be|\be)=1$. In this section we normalize the form so that $(\a|\a)=2$ and $(\be|\be)=-1/2$. With this normalization $k=-3/4$. As in \cite{AKMPP-JA}, we let $spo(2|3)$ denote $\g$ with this latter choice of the invariant form. Different normalizations occur in the literature: for example in \cite{ C-RIMS} the form is chosen so that $(\a|\a)=-8$ and $(\be|\be)=2$, hence $k=3$.

We have a vertex algebra homomorphism  $\Phi : V ^ {k}  (sl(2)) \otimes  V ^ {-4k} (sl(2))\rightarrow V_{k} (spo(2|3)) $.

 \begin{lemma} \label{singular}
 There are no  $\widehat \g_{\bar 0}$--singular vectors in $V_{-3/4} (spo(2|3))$  of $\g_{\bar 0}$--weights
 $$ (8\omega_1, 0), (7\omega_1, 2 \omega_1), ( 6\omega_1, 2\omega_1), (5\omega_1, 0), (0, 8 \omega_1), (\omega_1 , 6\omega_1). $$
 \end{lemma}
 \begin{proof}
 Let  $v_{n,m}$  be the space of $\widehat \g_{\bar 0}$--singular vectors in $V_{-3/4} (spo(2|3))$  of $\g_{\bar 0}$--weight $(n \omega_1, m \omega_2)$. 
 Let $V_{n,m} = \mathcal V_k(\g_{\bar 0})\cdot v_{n,m}$. The fusion rules argument and Clebsch-Gordan formulas (see e.g. \cite[\S 22]{Hum}) imply that 
\bea  \label{sum} V_{n_1,m_1} \cdot V_{n_2, m_2} \subset \sum_{i= 0} ^{ \mbox{min} \{ n_1, n_2\}}   \sum_{j= 0} ^{ \mbox{min} \{ m_1, m_2\}}  V_{n_1 + n_2 - 2 i, m_1 + m_2 -2j}. \eea
We can exclude summands $V_{r,s}$  in (\ref{sum})  such that the conformal weight of $v_{r,s}$, i.e.
$$h_{r,s} = \frac{ r  (r+2)} {5} +  \frac{ s (s+ 2)} {20},$$
is not an integer.  

Assume first that  
$v_{0,8}\ne\{0\}$. 
The fusion rules 
\bea   V_{1,2} \cdot V_{0, 8 } \subset  V_{1,10} + V_{1,8} +  V_{1,6}, \nonumber   \eea
and 
\bea   h_{1,10} = 33/5,  h_{1,8} = 23/5,  h _{1,6} = 3,  \nonumber   \eea
imply that $V_{-3/4}(spo(2|3))$ must contain a  $\widehat \g_{\bar 0}$--singular vector $v$ of  $\g_{\bar 0}$ weight $(\omega_1,6 \omega_1)$.

Using
\bea   V_{1,2} \cdot V_{1, 6} \subset  V_{2,8} + V_{2,6} +  V_{2,4} + V_{0,8} + V_{0,6} +  V_{0,4} , \nonumber   \eea
and
\bea   h_{2,8} = 28/5,  h_{2,6} = 4, h _{2,4} = 14/5,  h_{0,8} = 4,  h_{0,6} = 12/5,   h_{0,4} = 6/5,  \nonumber   \eea
we get that $v$ is  $\widehat {\g}$--singular. A contradiction, since $V_{k}(\g)$ is simple.

In this way we have proved that there are no  $\widehat \g_{\bar 0}$-- singular vectors of weights $ (\omega_1, 6\omega_1)$, $(0, 8\omega_1)$.

 Since   the maximal ideal  of $  V ^ {3} (sl(2))$ is generated by a  singular vector  of $\g_{\bar 0}$-weight  $(0,8\omega_1)$, we also have that $\mathcal V ^ {3} (sl(2))= V _ {3} (sl(2)).$
In particular, we can refine the fusion rule information from (\ref{sum}) by using fusion rules for $V_3 (sl(2))$ and get
\bea  
&&   V_{1,2} \cdot V_{5, 0} \subset  V_{6,2} + V_{4,2}, \nonumber \\ 
&&  V_{1,2} \cdot V_{6, 2} \subset    V_{7,2}+V_{7,0}+ V_{5,2} + V_{5, 0},  \nonumber  \\
&&  V_{1,2} \cdot V_{7, 2} \subset    V_{8,2}+V_{8,0}+ V_{6,2} + V_{6, 0},  \nonumber  \\
&&  V_{1,2} \cdot V_{8, 0} \subset    V_{9,2}+V_{7,2}.  \nonumber  
\eea
Since  the only integral values of conformal weights associated with the above decompositions are  $h_{8,0} = 16, h_{7,2} = 13, h_{6,2} = 10, h_{5,0} = 7$, we get 
\bea  
&&   V_{1,2} \cdot V_{5, 0} \subset  V_{6,2}, \label{5,0} \\ 
&&  V_{1,2} \cdot V_{6, 2} \subset    V_{7,2}+ V_{5, 0},  \label{6,2}  \\
&&  V_{1,2} \cdot V_{7, 2} \subset    V_{8,0}+ V_{6,2},  \label{7,2}  \\
&&  V_{1,2} \cdot V_{8, 0} \subset    V_{7,2}.  \label{8,0}
\eea

The remaining assertions of the Lemma can be obtained by using the following arguments: 

\begin{itemize}
 \item[(a)] (\ref{5,0}) implies that if $v$ is a $\widehat \g_{\bar 0}$-- singular vectors of weight $ (5\omega_1, 0)$, then $v$ must be $\widehat {\g}$-- singular.  A contradiction.
 
\item[(b)]  (\ref{6,2}) implies that if $v$ is a non-trivial $\widehat \g_{\bar 0}$-- singular vector of weight $ (6\omega_1, 2 \omega_1)$, then  there is a non-trivial  $\widehat \g_{\bar 0}$-- singular vector of  weight $ (5\omega_1, 0)$. A contradiction because of (a).

 \item[(c)] (\ref{7,2}) implies that if $v$ is a non-trivial $\widehat \g_{\bar 0}$-- singular vector of weight $ (7\omega_1, 2 \omega_1)$, then  there is a non-trivial  $\widehat \g_{\bar 0}$-- singular vector of  weight $ (6\omega_1, 2\omega_1)$. A contradiction because of (b). 
 
 \item[(d)]  (\ref{8,0}) implies that if $v$ is a non-trivial $\widehat \g_{\bar 0}$-- singular vector of weight $ (8\omega_1, 0)$, then  there is a non-trivial  $\widehat \g_{\bar 0}$-- singular vector of  weight $ (7\omega_1, 2\omega_1)$. A contradiction because of (c).
 \end{itemize} 
 \end{proof}

Let $e_{nm},f_{nm}$ be as in \S\ 8.5 of \cite{KW}: $e_{nm}$ is a root vector for the root $n\a_1+m\a_2$, and $f_{nm}$ is a root vector for the root $-(n\a_1+m\a_2)$. If $\{h_1,h_2\}$ is a basis for $\h$, then a basis of $\spo$ is
$$
\mathcal B=\{e_{22},e_{12},e_{11},e_{10},e_{01},h_{1},h_{2},f_{22},f_{12},f_{11},f_{10},f_{01}\}.
$$

Moreover, up to a renormalization of the generators, we have
\begin{equation}\label{strconst}
 [e_{22},e_{12}]=[e_{01},e_{12}]=0,\ [f_{22},e_{12}]=-f_{10},
 \end{equation}
\begin{equation}\label{strconst2}
[f_{10},e_{12}]=0,\ [f_{01},e_{12}]=-(1/2)e_{11},\ [e_{11},e_{12}]=0.
 \end{equation}
  \begin{equation}\label{strconst3}
[e_{22},f_{22}]=h_1,\ [h_{1},e_{01}]=0,\ [f_{22},e_{11}]=f_{11},\ [f_{22},e_{12}]=-f_{10}.
 \end{equation}
   \begin{equation}\label{strconst4}
[f_{11},e_{12}]=-1/2e_{01},\ [e_{01},f_{01}]=h_2,\ [f_{01},e_{11}]=-(1/2)e_{10}.
 \end{equation}
    \begin{equation}\label{strconst5}
[e_{10},e_{12}]=e_{22},\ [e_{01},e_{11}]=e_{12} ,\ [e_{11},e_{11}]=-e_{22}.
\end{equation}
If $V$ is a vertex algebra and $a,b\in V$, denote by $:a b: = a(-1)b$ their normal order. 
\begin{lemma}\label{30isnonzero} We have
$$:e_{10}e_{11}e_{12}:\notin V_{1,2}.$$
\end{lemma}
\begin{proof}
A basis of the space of vectors in $V^k(\spo)$ of $sl(2)\times sl(2)$--weight $(3\omega_1,0)$ and of conformal weight 3 is given by
\begin{align*}
\mathcal C=\{:T(e_{22})e_{11}:,&:T(e_{11})e_{22}:,:f_{11}e_{22}e_{22}:,:e_{22}e_{01}e_{10}:,\\&:f_{01}e_{22}e_{12}:,:h_{1}e_{22}e_{11}:,:h_{2}e_{22}e_{11}:,:e_{10}e_{11}e_{12}:\}.
\end{align*}
If $v$ is in the span of $\mathcal C$  and in the maximal ideal of $V^k(\spo)$, then  $x(1)y(1)v=0$ for all $x,y\in\spo$. By computing $x(1)y(1)v$ with $x,y\in \mathcal B$ and $v$ a generic linear combination of elements of $\mathcal C$, if $v$ is in the maximal ideal of $V^k(\spo)$, 
\begin{align}\label{inmax}
v\in \C(:T(e_{22})e_{11}:&-3/4:T(e_{11})e_{22}:-:f_{11}e_{22}e_{22}:-1/2:e_{22}e_{01}e_{10}:\\\notag
&+:f_{01}e_{22}e_{12}:-:h_{1}e_{22}e_{11}:).
\end{align}

If $:e_{10}e_{11}e_{12}:\in V_{1,2}$ then it is a linear combination in $V_k(\spo)$ of $\mathcal C\setminus\{:e_{10}e_{11}e_{12}:\}$, but this implies that $:e_{10}e_{11}e_{12}:$ plus a linear combination of elements of $\mathcal C\setminus\{:e_{10}e_{11}e_{12}:\}$ belongs to the maximal ideal of $V^k(\spo)$, and this contradicts \eqref{inmax}.
\end{proof}

  \begin{proposition}   \label{osp(3|2)-dec}
(1). The vertex algebra $ V_{-3/4} (sl(2)) \otimes V_3 (sl(2) )$ is conformally embedded into  $V_{-3/4}(spo(2|3))$.\par\noindent
(2).    The following decomposition holds
\bea  V_{-3/4}(spo(2|3))  & =&   ( V_{-3/4} (sl(2))  \oplus     L_{sl(2)} (3 \omega_1)  )\otimes V_3 (sl(2) )  \bigoplus \nonumber  \\ &&  (  L_{sl(2)} ( \omega_1) \oplus  L_{sl(2)} (2 \omega_1)  ) \otimes L_{sl(2)}  (2\omega_1). \nonumber \eea
 \end{proposition}
 \begin{proof}
Since  $ V ^ {-3/4} (sl(2))$  (resp. $  V ^ {3} (sl(2))$) contains a unique singular vector of $\g_{\bar 0}$-weight $(8 \omega_1, 0)$ (resp. $(0,8\omega_1)$),  Lemma \ref{singular} implies that   $ \mathcal V_ {-3/4} (sl(2))$ and $ \mathcal V_{3} (sl(2))$ are simple vertex algebras. This proves (1).

\medskip

\noindent(2) 
First we notice that $V_{-3/4}(spo(2|3))$ is semisimple as a module for its subalgebra $V_{-3/4} (sl(2)) \otimes V_3 (sl(2)) $ (we use the  facts that $V_{-3/4} (sl(2))$ is rational in the category $\mathcal{O}$ \cite{AdM} and that $V_3(sl(2))$ is a rational vertex operator algebra).
It is very easy to check that $:e_{10}e_{11}e_{12}:$ is a $\widehat \g_{\bar 0}$--singular vector in $V_k(\spo)/V_{1,2}$. By Lemma \ref{30isnonzero}, $:e_{10}e_{11}e_{12}:$ is nonzero in $V_k(\spo)/V_{1,2}$. Thus, by semisimplicity, $V_{3,0}$ is nonzero in $V_k(\spo)$. Moreover, since $:e_{10}e_{11}e_{12}:$ is the unique element of $\mathcal C$ which is not in $V_{1,2}$, we see that $V_{3,0}\simeq L_{sl(2)}(3\omega_1)\otimes V_3(sl(2))$.

By fusion rules, we see that $V_{1,2}\cdot V_{3,0}=V_{2,2}$.
 
The subspace of $ V^k(\g)$ of vectors having conformal weight $2$ and $\h$-weight $(2\omega_1,2\omega_1)$ has basis $\{:e_{22}e_{01}:,:e_{11}e_{12}:\}$.
Using \eqref{strconst3}, we find
$$
[(f_{22})_\l:e_{22}e_{01}:]=-:h_1e_{01}:-(3/4)\l e_{01}
$$
$$
[(f_{22})_\l:e_{11}e_{12}:]=:f_{11}e_{12}:-:e_{11}f_{10}:-(1/2)\l e_{01}.
$$

This implies that, up to a constant, there is only one $\widehat \g_{\bar 0}$--singular vector of weight $(2\omega_1,2\omega_1)$. Since $V_{-3/4}(spo(2|3))$ is semisimple as  $V_{-3/4} (sl(2)) \otimes V_3 (sl(2)) $--module, it follows that $V_{2,2}\simeq L_{sl(2)}(2\omega_1)\otimes L_{sl(2)}(2\omega_1)$. 

Note that, by Lemma \ref{singular}, we have the following fusion rules
\bea
 && V_{1,2} \cdot V_{1,2} \subset V_{0,0} + V_{2,2},\nonumber \\
 &&V_{1,2}\cdot  V_{2,2}\subset  V_{3,0}+ V_{1,2},  \nonumber \\
    &&V_{1,2}\cdot V_{3,0}\subset V_{2,2}, \nonumber \\
&& V_{2,2} \cdot V_{2,2} \subset V_{0,0} + V_{2,2}, \label{fibonacci}
  \nonumber \\
  &&V_{2,2}\cdot V_{3,0}\subset  V_{3,0}+V_{1,2},\nonumber \\
    &&V_{3,0}\cdot V_{3,0}\subset V_{0,0}. \nonumber
 \eea
\noindent Thus $ \mathcal U = V_{0,0} \oplus V_{1,2} \oplus V_{3,0} \oplus V_{2,2} $ 
is a vertex subalgebra of $V_{-3/4}(spo(2|3))$.
 Since $\g \subset \mathcal U$, we get $\mathcal U = V_{-3/4}(spo(2|3))$. 
 \end{proof}

\begin{remark}
The decomposition in Proposition  \ref{osp(3|2)-dec} has recently  also appeared in the lecture notes of T. Creutzig \cite{C-RIMS} presented at RIMS. In the proof of decomposition he uses some very non-trivial result 
on the extension theory of vertex operator algebras based on vertex tensor categories. 

We should mention that our approach uses neither tensor product theory nor  extension theory of vertex algebras. 
It would be interesting to understand how the  tensor category approach imposes further constraints on the dot product structure and possibly makes our approach more effective.
\end{remark}
\subsection {The case  ${\mathfrak g}=C(n+1)$, $k=1$}
\label{CCn1}

Let $M_{(m|2n)}$ be the vertex algebra introduced in Section \ref{FFR} below. Here we specialize to the case $m=2$. In particular we let $V$ be the superspace $\C^{(2|2n)}$ with reversed parity. 

The vertex algebra $M_{(2|2n)}$ is isomorphic to $F_{(1)}\otimes M_{(n)}$, where $F_{(1)}$ is the fermionic vertex algebra generated by $V_{\bar 1}$ equipped with the symmetric form $\langle\cdot,\cdot\rangle_{|V_{\bar 1}}$ and $M_{(n)}$ is the Weyl vertex algebra generated by $V_{\bar 0}$  equipped with the symplectic form $\langle\cdot,\cdot\rangle_{|V_{\bar 0}}$.
By the boson-fermion correspondence \cite{KacV} $F_{(1)}\cong V_L$ where $V_L = M_\alpha(1) \otimes {\C}[L]$ is the lattice vertex algebra  associated to the lattice $L={\Z}\alpha$, $\langle \alpha , \alpha \rangle = 1$.  We have
$ V_L = V_L ^{0} \oplus V_L ^{1}$, where $V_L ^{0} $ ( resp.  $V_L ^{1} ) $  is the even part  (resp. odd part)   of $V_L$.
Moreover,
$$ V_{L} ^0 = M_\a(1) \otimes {\C}[\Z (2 \alpha)], \quad V_{L} ^1 = M_\a(1) \otimes {\C}[\alpha + \Z (2 \alpha)].$$

The conformal vector in $M_\a(1) \subset V_{L} ^0 \subset F_{(1)}$ is $\omega_F = \frac{1}{2} : \alpha \alpha:$.
\begin{proposition}\label{decCn1}
\item[(1) ]   There is a conformal embedding $V_{1} (\g) \rightarrow M_{(2|2n)} $ uniquely determined by
\begin{equation}\label{embeddingM}
X\mapsto \half \sum_{i}:X(e_i)e^i:, \quad X\in osp(2|2n).
\end{equation}

\item[(2)] There is a conformal embedding of  $V_{-1/2} (sp(2n) ) \otimes  V_L ^0 $  in $V_{1} (\g)$ and  we have the following decomposition
$$ V_{1} (\g)  = V_{-1/2} (sp(2n) ) \otimes  V_L ^0 \bigoplus L_{sp(2n)} ( \omega_1) \otimes V_L ^1. $$
\end{proposition}
\begin{proof}\ 

\noindent (1) The fact that \eqref{embeddingM} extends to a map from   $V^{1} (\g)$ to $M_{(2|2n)}$ and that the image is simple is given in Theorem 7.1 of \cite{Kw}. We provide an alternative proof in Section \ref{FFR} .  The check that the emebdding is conformal is given in Lemma \ref{isconformal} below.

\noindent (2) Let $M_{(2|2n)}^\pm$ be as in Section \ref{FFR}. By Theorem 7.1 of \cite{Kw}, $M_{(2|2n)}^+\simeq V_1(osp(2|2n))$. Clearly
$$
M_{(2|2n)}^+=M_{(n)}^+\otimes F_{(1)}^+\oplus M_{(n)}^-\otimes F_{(1)}^-=M_{(n)}^+\otimes V_L ^0\oplus M_{(n)}^-\otimes V_L ^1.$$
one has $M_{(n)}^+\simeq V_{-1/2}(sp(2n))$ and $M_{(n)}^-\simeq L_{sp(2n)}(\omega_1)$. 
\end{proof}
 
 \begin{remark}
 The decomposition in Proposition \ref{decCn1} (2) is the eigenspace decomposition of $V_1(\g)$ for the involution induced by the parity involution of $\g$. Indeed, it is enough to verify that, if $X\in \g_{\bar 0}$, then $X(-1)\vac\in M_{(n)} ^+\otimes F_{(2)}^+$ and, if $X\in \g_{\bar 1}$, then $X(-1)\vac\in  M_{(n)} ^- \otimes F_{(1)}^-$. This follows from \eqref{embeddingM}.
 \end{remark}

\subsection{The case  ${\mathfrak g}=F(4)$, $k=1$}

 \begin{lemma}
 The vertex subalgebra of $V_1(\g)$ generated by $\g_{\bar 0}$ is simple and isomorphic to  $V_1 (so(7)) \otimes V_{-\frac{2}{3}} (sl(2))$.
 \end{lemma}

\begin{proof} 
The vertex subalgebra of $V_1(\g)$ generated by $\g_{\bar 0}$ is isomorphic to $V_1 (so(7)) \otimes \mathcal V_{-2/3} (sl(2))$ where  $\mathcal V_{-2/3} (sl(2))$  is either simple or universal affine vertex algebra associated to $sl(2)$ at level $-2/3$. Similarly, the $V_1 (so(7)) \otimes \widetilde V_{-2/3} (sl(2))$--module generated by $\g_{\bar 1}$ is isomorphic to
  $L_{so(7)} ( \omega_3) \otimes  \widetilde  L_{sl(2)} (  \omega _1)  $, where $\widetilde  L_{sl(2)} (\omega_1)  $ is a highest weight $  \mathcal V_{-2/3} (sl(2))$--module, of $sl(2)$--highest weight  $\omega_1$.
  
We let  $v_{\l,\mu}$ be the set of $\widehat\g_{\bar0}$--singular vectors of $\g_{\bar0}$--weight $(\l,\mu)$ and $V_{\l,\mu}=\mathcal V_1(\g_{\bar0})\cdot v_{\l,\mu}$. Let also $h_{\l,\mu}$ be the conformal weight of a vector $v\in v_{\l,\mu}$.

 Assume that  $\mathcal V_{-2/3} (sl(2)) = V^{-2/3} (sl(2))$. Then it has a unique singular vector $\Omega_0$  of $sl(2)$--weight $6  \omega_1$, thus $V_{0,6\omega_1}\ne\{0\}$.

  By using  the tensor product decomposition
 $$ V_{sl(2)} (\omega_1) \otimes V_{sl(2)} (6 \omega_1) = V_{sl(2)} ( 5 \omega_1) \oplus V_{sl(2)} (7 \omega_1), $$
 we see that
 $$
 V_{0,6\omega_1}\cdot   V_{\omega_3,\omega_1}  \subset V_{\omega_3,5\omega_1}+V_{\omega_3,7\omega_1}.
 $$ 
Since $h_{\omega_3,7\omega_1}=49/4$ , we see that $V_{\omega_3,5\omega_1}\ne\{0\}$. 
  Next we use the decomposition
 $$ V_{sl(2)} (\omega_1) \otimes V_{sl(2)} (5 \omega_1) = V_{sl(2)} ( 6 \omega_1) \oplus V_{sl(2)} (4 \omega_1) $$ to deduce that
$$ V_{\omega_3,\omega_1}\cdot V_{\omega_3,5\omega_1}\subset V_{0,6\omega_1}+V_{0,4\omega_1}+V_{\omega_1,6\omega_1}+V_{\omega_1,4\omega_1}.$$
Since $h_{0,4\omega_1}=9/2$, $h_{0,6\omega_1}=9$, $h_{\omega_1,4\omega_1}=5$, $h_{\omega_1,6\omega_1}=19/2$, and $h_{\omega_3,5\omega_1}=7$, we see that $V_{\omega_1,4\omega_1}\ne\{0\}$ otherwise any $v\in v_{\omega_3,5\omega_1}$ is $\widehat \g$--singular.

By using decomposition 
 $$ V_{sl(2)} (\omega_1) \otimes V_{sl(2)} (4 \omega_1) = V_{sl(2)} ( 5 \omega_1) \oplus V_{sl(2)} (3 \omega_1), $$
 and fusion rules of $V_1(so(7))$--modules
 $$ L_{so(7)} ( \omega_1)  \times L_{so(7)} ( \omega_3)   = L_{so(7)} ( \omega_3)  $$
 we conclude that 
 $$ V_{\omega_3,\omega_1}\cdot V_{\omega_1,4\omega_1}\subset V_{\omega_3,5\omega_1}+V_{\omega_3,3\omega_1}.$$
 
 Since $h_{\omega_3,3\omega_1}=13/4$, $h_{\omega_3,5\omega_1}=7$, and $h_{\omega_1,4\omega_1}=5$,  we conclude that any  $v\in v_{\omega_1,4\omega_1}$ is $\widehat{\g}$--singular.   This is in contradiction with  the simplicity of $V_1(\g)$. 
 
 Therefore $\mathcal V_{-2/3}(sl(2)) = V_{-2/3} (sl(2))$ and the claim follows.
\end{proof}

In this section we follow the description of the roots of $\g$ given in \cite{KacReps}.

\begin{lemma}\label{struct_const}
The following formulas hold in $V^1(F(4))$:
\begin{enumerate}
\item $[(x_{\e_1-\e_2})_\l :x_\d x_{\e_1}:]=0$,
\item $[(x_{\e_1-\e_2})_\l :x_{1/2(\d+\e_1+\e_2+\e_3)} x_{1/2(\d+\e_1-\e_2-\e_3)}:]=0$,
\item $[(x_{\e_1-\e_2})_\l :x_{1/2(\d+\e_1+\e_2-\e_3)} x_{1/2(\d+\e_1-\e_2+
\e_3)}:]=0$,
\item $[(x_{\e_2-\e_3})_\l :x_\d x_{\e_1}:]=0$,
\item $[(x_{\e_2-\e_3})_\l :x_{1/2(\d+\e_1+\e_2+\e_3)} x_{1/2(\d+\e_1-\e_2-\e_3)}:]=0$,
\item $[(x_{\e_2-\e_3})_\l :x_{1/2(\d+\e_1+\e_2-\e_3)} x_{1/2(\d+\e_1-\e_2+
\e_3)}:]=0$,
\item $[(x_{\e_3})_\l :x_\d x_{\e_1}:]=:x_\d x_{\e_1+\e_3}:$,
\item $[(x_{\e_3})_\l :x_{1/2(\d+\e_1+\e_2+\e_3)} x_{1/2(\d+\e_1-\e_2-\e_3)}:]=$\newline
$-:x_{1/2(\d+\e_1+\e_2+\e_3)}x_{1/2(\d+\e_1-\e_2+\e_3)}:$,
\item $[(x_{\e_3})_\l :x_{1/2(\d+\e_1+\e_2-\e_3)} x_{1/2(\d+\e_1-\e_2+
\e_3)}:]=$\newline
$:x_{1/2(\d+\e_1+\e_2+\e_3)}x_{1/2(\d+\e_1-\e_2+\e_3)}:$,
\item $[(x_{\d})_\l :x_\d x_{\e_1}:]=0$,
\item $[(x_{\d})_\l :x_{1/2(\d+\e_1+\e_2+\e_3)} x_{1/2(\d+\e_1-\e_2-\e_3)}:]=0$,
\item $[(x_{\d})_\l :x_{1/2(\d+\e_1+\e_2-\e_3)} x_{1/2(\d+\e_1-\e_2+
\e_3)}:]=0$,
\item $[(x_{-\e_1-\e_2})_\l :x_\d x_{\e_1}:]=-2:x_\d x_{-\e_2}:$,
\item $[(x_{-\e_1-\e_2})_\l :x_{1/2(\d+\e_1+\e_2+\e_3)} x_{1/2(\d+\e_1-\e_2-\e_3)}:]=$\newline
$-2:x_{1/2(\d-\e_1-\e_2+\e_3)}x_{1/2(\d+\e_1-\e_2+\e_3)}:$,,
\item $[(x_{-\e_1-\e_2})_\l :x_{1/2(\d+\e_1+\e_2-\e_3)} x_{1/2(\d+\e_1-\e_2+
\e_3)}:]=0$\newline
$2:x_{1/2(\d-\e_1-\e_2+\e_3)}x_{1/2(\d+\e_1-\e_2+\e_3)}:$,
\item $[(x_{-\d})_\l :x_\d x_{\e_1}:]=-:h_\d x_{\e_1}:-2/3\l x_{\e_1}$,
\item $[(x_{-\d})_\l :x_{1/2(\d+\e_1+\e_2+\e_3)} x_{1/2(\d+\e_1-\e_2-\e_3)}:]=$\newline
$:x_{1/2(-\d+\e_1+\e_2+\e_3)}x_{1/2(\d+\e_1-\e_2-\e_3)}:$\newline$+:x_{1/2(\d+\e_1+\e_2+\e_3)}x_{1/2(-\d+\e_1-\e_2-\e_3)}:+8/3\l x_{\e_1}$,
\item $[(x_{-\d})_\l :x_{1/2(\d+\e_1+\e_2-\e_3)} x_{1/2(\d+\e_1-\e_2+
\e_3)}:]=$\newline
$:x_{1/2(-\d+\e_1+\e_2-\e_3)}x_{1/2(\d+\e_1-\e_2+\e_3)}:$\newline$+:x_{1/2(\d+\e_1+\e_2-\e_3)}x_{1/2(-\d+\e_1-\e_2+\e_3)}:-8/3\l x_{\e_1}$,
\item $[(x_{-1/2(\d+\e_1+\e_2+\e_3)})_\l :x_{1/2(\d+\e_1+\e_2+\e_3)} x_{1/2(\d+\e_1-\e_2-
\e_3)}:]=$\newline
$-16/3:h_{1/2(\d+\e_1+\e_2+\e_3)}x_{1/2(\d+\e_1-\e_2-\e_3)}:$\newline$-8/3:x_{1/2(\d+\e_1+\e_2+\e_3)}x_{-\e_2-\e_3)}:+32/3\l x_{1/2(\d+\e_1-\e_2-\e_3)}$,
\item $[(x_{-1/2(\d+\e_1+\e_2+\e_3)})_\l :x_{1/2(\d+\e_1+\e_2-\e_3)} x_{1/2(\d+\e_1-\e_2+
\e_3)}:]=$\newline
$8/3:x_{-\e_3}x_{1/2(\d+\e_1-\e_2+\e_3)}:$\newline$-8/3:x_{1/2(\d+\e_1+\e_2-\e_3)}x_{-\e_2)}:+8/3\l x_{1/2(\d+\e_1-\e_2-\e_3)}$.
\end{enumerate}
\end{lemma}
\begin{proof}
We apply Wick's formula and an explicit calculation of the structure constants for $F(4)$ following \cite{DeW}.
\end{proof}

\begin{lemma} \label{sing-two-f4} In $V^1(F(4))$ the unique (up to a multiplicative constant) $\widehat \g_{\bar 0}$--singular vector of conformal weight $2$ and $\h$--weight $\d+\e_1$ is
$$
v_{\d+\e_1}\!=:x_{1/2(\d+\e_1+\e_2+\e_3)} x_{1/2(\d+\e_1-\e_2-\e_3)}:\!+\!:x_{1/2(\d+\e_1+\e_2-\e_3)} x_{1/2(\d+\e_1-\e_2+
\e_3)}:
$$
Moreover (the image of) $v_{\d+\e_1}$ is nonzero in $V_1(F(4))$.
\end{lemma}

\begin{proof}By Lemma \ref{struct_const}, one checks that $v_{\d+\e_1}$ is $\widehat \g_{\bar 0}$--singular. To check that it is the only one, we observe 
that
a basis of the space of vectors in $V^1(F(4))$ of conformal weight $2$ and $\h$--weight $\d+\e_1$ is $\{v_1,v_2,v_3\}$ where
$$v_1=:x_\d x_{\e_1}:,\quad v_2=:x_{1/2(\d+\e_1+\e_2+\e_3)} x_{1/2(\d+\e_1-\e_2-\e_3)}:,$$ 
$$v_3=:x_{1/2(\d+\e_1+\e_2+\e_3)} x_{1/2(\d+\e_1-\e_2-\e_3)}:.
$$

If a linear combination $av_1+bv_2+cv_3$ is $\widehat \g_{\bar 0}$--singular, then, by Lemma \ref{struct_const} (7)--(9) and (16)--(18) we have
$$
a-b+c=0, -\tfrac{2}{3}(a-4b+4c)=0,
$$
hence $a=0$ and $b=c$.

By Lemma \ref{struct_const} (19) and (20), we have
$$
x_{-1/2(\d+\e_1+\e_2+\e_3)}(1)v_{\d+\e_1}=\tfrac{40}{3}x_{1/2(\d+\e_1+\e_2+\e_3)}\ne0.
$$
\end{proof}

\begin{theorem} \label{conj-f4}
We have:
\bea  V_1(\g) & =&  V_1 (so(7)) \otimes V_{-\tfrac{2}{3}} (sl(2)) \bigoplus L_{so(7)} ( \omega_3) \otimes  L_{sl(2)} ( \omega_1)   \nonumber \\ && \bigoplus   L_{so(7)} ( \omega_1) \otimes  L_{sl(2)} (  2 \omega_1). \eea
\end{theorem}
\begin{proof}

Observe that 
$V_1(\g)$ is completely reducible as $V_1 (so(7)) \otimes V_{-\frac{2}{3}} (sl(2))$--module.
 Clearly $V_{\omega_3,\omega_1}\simeq L_{so(7)} ( \omega_3) \otimes  L_{sl(2)} (\omega_1)$ and, by Lemma \ref{sing-two-f4}, $$V_{\omega_1,2\omega_1}\simeq L_{so(7)} ( \omega_1) \otimes  L_{sl(2)} (2\omega_1).$$
For the proof it is enough to check that $V_{0,0}+V_{\omega_3,\omega_1}+V_{\omega_1,\omega_2}$ is a vertex subalgebra. This follows from the subsequent remarks.
\begin{itemize}
\item Since  $h_{0,2\omega_1}=3/2$ and $h_{\omega_1,0}=1/2$, 
$$
V_{\omega_3,\omega_1}\cdot V_{\omega_3,\omega_1}\subset V_{0,0}+V_{\omega_1,2\omega_1}.
$$
\item  Since $h_{\omega_3,3\omega_1}=13/4$, 
$$
V_{\omega_3,\omega_1}\cdot V_{\omega_2,2\omega_1}\subset V_{\omega_3,\omega_1}.
$$
\item Since $h_{0,4\omega_1}=9/2$ and  $h_{0,2\omega_1}=3/2$,
$$
V_{\omega_1,2\omega_1}\cdot V_{\omega_1,2\omega_1}\subset V_{0,0}.
$$
\end{itemize}
\end{proof}

\begin{remark}
The decomposition in Theorem \ref{conj-f4}  has also  appeared  in   \cite{C-RIMS}. 
\end{remark}
 \subsection{The case  ${\mathfrak g}=G(3)$, $k=1$ }

In this case $\g_{\bar 0}=(\g_{\bar 0})_1\oplus (\g_{\bar 0})_2$ with $(\g_{\bar 0})_1\simeq sl(2)$ and $(\g_{\bar 0})_2$ of type $G_2$.

 \begin{lemma} \label{singular-G3}
 There are no  $\widehat \g_{\bar 0}$--singular vectors in $V_{1} (G(3))$  of $\g_{\bar 0}$--weight
 $ (8 \omega_1, 0)$.
  The vertex subalgebra of $V_1(\g)$ generated by $\g_{\bar 0}$ is  isomorphic to  $ V_{-\frac{3}{4}} (sl(2))\otimes V_1 (G_2)$.
 \end{lemma}
 \begin{proof}

   The vertex subalgebra of $V_1(\g)$ generated by $\g_{\bar 0}$ is isomorphic to $\mathcal V_{-3/4} (sl(2))\otimes V_1 (G_2)$  where  $\mathcal V_{-3/4} (sl(2))$  is a quotient of $V^{-3/4}(sl(2))$. Indeed, the maximal ideal of $V^1(G_2)$ is generated by $:x_\theta x_\theta:$ where $\theta$ here is the highest root of $G_2$.  But  $:x_\theta x_\theta:$ is $\widehat \g_{\bar 0}$--singular and $\frac{(2\theta,2\theta+2\rho^2)_2}{2(1+h^\vee_2)}\ne 2$. 
   
 By Theorem 5.3 of \cite{AKMPP-JJM}, $\mathcal V_{-3/4} (sl(2))$ is either  the universal  or simple affine vertex algebra associated to $sl(2)$ at level $-3/4$ and the maximal ideal in  $V^{-3/4} (sl(2))$ is generated by a  unique singular vector of $sl(2)$--weight  $8 \omega_1$. Let us now show that such singular vector cannot exist.
  
 Let  $v_{n,m}$  be a the set of $\widehat \g_{\bar 0}$ singular vector in $V_{1} (G(3))$  of $\g_{\bar 0}$ weight $(n \omega_1, m \omega_2)$, where $ n \in {\Z}_{\ge 0}$ and $ m \in \{ 0, 1\}$.
 Let $V_{n,m} = \mathcal V_1(\g_{\bar0})\cdot v_{n,m}$. The fusion rules argument implies that 
\bea  V_{n , 0 } \cdot V_{1, 1} &\subset & V_{n+1, 1} + V_{n-1,1} . \nonumber \\
  V_{n , 1 } \cdot V_{1, 1} &\subset & V_{n+1, 1} + V_{n-1,1}  + V_{n+1,0} + V_{n-1, 0} . \nonumber \eea
We can exclude summands $V_{r,0}$     such that the conformal weight 
$$h_{r,0} = \frac{ r  (r+2)} {5}$$ of $v_{r,0}$ is not in $\Z_+$
and summands $V_{r,1}$  such that  the conformal weight 
$$ h_{r,1} =  \frac{ r  (r+2)} {5} + \frac{2}{5} $$ 
 of $v_{r,1}$ is not in $\Z_+$.

The only integral conformal weights in the above decompositions are
$$ h_{8,0} = 16, \ h_{5,0} = 7, \ h_{3,0} = 3, \  h_{0,0} = 1, \   h_{7,1} = 13, \ h_{6 ,1} =  10, \ h_{2,1} = 2. $$ 
It follows that
\bea
&& V_{8,0} \cdot V_{1,1} \subset V_{7,1}, \nonumber  \\
&& V_{7,1} \cdot V_{1,1} \subset V_{6,1} + V_{8,0},  \nonumber \\
&& V_{6,1} \cdot V_{1,1} \subset  V_{5,0} + V_{7,1},  \nonumber \\
&& V_{5,0} \cdot V_{1,1} \subset V_{6,1},\nonumber
\eea
and this 
implies that   $V_{8, 0}$  generates a proper ideal in $V_1(\g)$. A contradiction.
    This implies that
 $\mathcal V _ {-3/4} (sl(2))= V _ {-3/4} (sl(2)) $.
  The claim follows.
 \end{proof}
 
The next result is obtained as a consequence of the results of Section \ref{super-super-section} below, thus we postpone its proof to the end of \S\  \ref{ospinG(3)}.

\begin{theorem}\label{G2inG(3)}
We have
\bea V_1 (\g)  &=&  \left( V_{-\frac{3}{4}} (sl(2))   \oplus L_{sl(2)} (3 \omega_1)  \right) \otimes  V_1 (G_2) \nonumber \\
&& \bigoplus     \left(  L_{sl(2)} (  \omega_1)  \oplus L_{sl(2)} (2 \omega_1) \right) \otimes L_{G_2} (\omega_1) \nonumber \eea
\end{theorem}

  \section{Some examples of decompositions of  embeddings $\g^0\subset\g$} \label{super-super-section}
 \subsection{The conformal embedding  $gl(n|m)\hookrightarrow sl(n+1|m)$} \

Recall that $V_k(\g)^{(q)}$ is the eigenspace for the action of $\varpi{(0)}$ on $V_k(\g)$ corresponding to the eigenvalue $q$.
 
 \begin{theorem} Assume that we are in the following cases:
 \begin{itemize}
 \item Conformal level $k=1$, $m \ne n+2$.
 \item Conformal level $k =- \frac{h^{\vee}}{2}= - \frac{n+1-m}{2}$,  $n \ne m+2$, $n \ne m+3$.
 \end{itemize}
 Then each $V_k(\g)^{(q)} $ is a simple $V_k(\g^0)$--module.
 \end{theorem}
 \begin{proof} 
 We have to decompose the tensor product of the two pieces  of $\g^1$.
 Observe that $\C^{n|m}\otimes (\C^{n|m})^*\cong gl(n,m)$, hence the desired decomposition is
 $$\left(V_{\C\varpi}(\epsilon)\otimes \C^{n|m}\right)\otimes \left( V_{\C\varpi}(-\epsilon)\otimes (\C^{n|m})^*\right)=\C\otimes\C\oplus \C\otimes sl(n,m).
 $$
(Recall from \eqref{epsilon} the definition of $\epsilon$). We can now  apply Theorem \ref{generalhs}. If $k=1$  formula \eqref{finioths} reads
$$\frac{-m+n}{1-m+n}=1-\frac{1}{1-m+n},$$
 which is never integral in our hypothesis. For $k = - \frac{n+1-m}{2}$ we obtain
 $$ \frac{2(m-n)}{1+m-n}=2-\frac{2}{1+m-n},$$
 which is again never integral in our hypothesis. The claim follows.
 \end{proof}
 
Using the free field realization of \cite{KWW} for $k=1$, we can  actually write down the decomposition and also cover the missing $m=n+2$ case. In $sl(n|m)$ set $\alpha_i^\vee=E_{ii}-E_{i+1i+1}$ for $i\ne n$ and $\alpha_n^\vee=E_{nn}+E_{n+1n+1}$. Define $\omega_i\in \h^*$ by setting $\omega_i(\alpha_j^\vee)=\delta_{ij}$ and $\omega_0=0$.
 Set
 $$
 \lambda_{(q)}=\begin{cases}\omega_q&\text{if $0\le q\le n$,}\\
(1+q-n) \omega_n+(q-n)\omega_{n+1}&\text{if $q\ge n$,}\\
-q\omega_{m+n-1}&\text{if $q\le 0$.}
\end{cases}
$$
\begin{proposition}
As a $M_{c}(1)\otimes  V_1(sl(n|m))$--module
$$
V_1(sl(n+1|m))=\sum_{q\in\ZZ} M_{c}(1,\tfrac{-q}{\Vert \varpi\Vert})\otimes L_{sl(n|m)}(\lambda_{(-q)}).
$$
\end{proposition}
\begin{proof}
Set $\varpi_1=I_{n+1+m},\varpi_2=E_{11},\varpi_3=\begin{pmatrix}0&0\\
0&I_{n+m}\end{pmatrix}\in gl(n+1|m)$ and let $c_i=\varpi_i/\Vert\varpi_i\Vert$. By \cite{KWW}, \S \ 3, there is an embedding of $M_{c_1}(1)\otimes V_1(sl(n+1|m)$ in $M_{(2n+2|2m)}$. The action of $\varpi_1(0)$ on $M_{(2n+2|2m)}$ defines the charge decomposition $M_{(2n+2|2m)}=\oplus_{q\in\ZZ}M_{(2n+2|2m)}^q $ and 
$$
M_{c_1}(1)\otimes V_1(sl(n+1|m)=M_{(2n+2|2m)}^0.
$$
In particular, if $M_{c_1}(1)^+=span(\varpi_1(n)\mid n>0)$,
$$
V_1(sl(n+1|m)=(M_{(2n+2|2m)}^0)^{M_{c_1}(1)^+}.
$$
Clearly, $M_{(2n+2|2m)}=M_{(2|0)}\otimes M_{(2n|2m)}$. By boson-fermion correspondence, as a $M_{c_2}(1)$--module,  $M_{(2|0)}=\sum_{q\in\ZZ}M_{c_2}(1,q)$. The action of $\varpi_3(0)$ on $M_{(2n|2m)}$ defines the charge decomposition $M_{(2n|2m)}=\oplus_{q\in\ZZ}M_{(2n|2m)}^q $ and, by \cite{KWW}, \S \ 3,
$$
M_{(2n|2m)}^q=M_{c_3}(1,\tfrac{q}{\Vert \varpi_3\Vert})\otimes L_{sl(n|m)}(\lambda_{(q)})
$$
as a $M_{c_3}(1)\otimes V_1(sl(n|m))$--module. Since $\varpi_1=\varpi_2+\varpi_3$,
$$
M_{(2n+2|2m)}^0=\sum_{q\in\ZZ}M_{(2|0)}^q\otimes M_{(2n|2m)}^{-q},
$$
so
$$
M_{(2n+2|2m)}^0=\sum_{q\in\ZZ}M_{c_2}(1,q)\otimes M_{c_3}(1,\tfrac{-q}{\Vert \varpi_3\Vert})\otimes L_{sl(n|m)}(\lambda_{(-q)})
$$
as a $M_{c_2}(1)\otimes M_{c_3}(1)\otimes V_1(sl(n|m))$--module.
Since $\varpi_1=\varpi_2+\varpi_3$ and $\varpi=\tfrac{m-n}{1+n-m}\varpi_2+\frac{1}{1+n-m}\varpi_3$, we obtain that 
$$
M_{c_2}(1,q)\otimes M_{c_3}(1,\tfrac{-q}{\Vert \varpi_3\Vert})=M_{c_1}(1,0)\otimes M_{c}(1,\tfrac{-q}{\Vert \varpi\Vert}).
$$
The final outcome is that
\bea
V_1(sl(n+1|m))&=&(\sum_{q\in\ZZ}M_{c_1}(1,0)\otimes M_{c}(1,\tfrac{-q}{\Vert \varpi\Vert})\otimes L_{sl(n|m)}(\lambda_{(-q)}))^{M_{c_1}(1)^+}\nonumber\\&=&\sum_{q\in\ZZ} M_{c}(1,\tfrac{-q}{\Vert \varpi\Vert})\otimes L_{sl(n|m)}(\lambda_{(-q)})\nonumber
\eea
as wished.
\end{proof}

  \subsection{The conformal embedding $sl(2) \times spo(2|3) \hookrightarrow G(3)$, $k=1$}\label{ospinG(3)}
 
 In this section we consider $\g= G(3)$ and its subalgebra $\g^0 = sl(2) \times spo(2|3)$. We will use the notation established in the proof of Theorem \ref{super-super} (2).

Recall that
 $\g ^1 =  V_{sl(2)} (\omega_1) \otimes V_{spo(2|3)}(\beta_1+3/2\beta_2)$.

In order to apply  Theorem \ref{general} we need to compute the factors occurring in the composition series of $V_{spo(2|3)}(\beta_1+3/2\beta_2)\otimes V_{spo(2|3)}(\beta_1+3/2\beta_2)$. Clearly $V_{spo(2|3)}(2\beta_1+3\beta_2)$ occurs. By looking at Table 3.65 of \cite{dictionary} one sees that $V_{spo(2|3)}(2\beta_1+3\beta_2)$ has dimension 30. Observe that $\dim V_{spo(2|3)}(\beta_1+3/2\beta_2)=8$ and its  $sp(2)\times so(3)$ decomposition is $V_{sp(2)}(\omega_1)\otimes V_{so(3)}(\omega_1)+ V_{sp(2)}(0)\otimes V_{so(3)}(3\omega_1)$, so $V_{sp(2)}(0)\otimes V_{so(3)}(6\omega_1)$ must occur in the tensor product. The only representation of dimension less that 34 where such a factor occurs is $V_{spo(2|3)}(\beta_1+3\beta_2)$ which has dimension 20. The remaining $sp(2)\times so(3)$--factors in the tensor product are 
$$V_{sp(2)}(2\omega_1)\otimes V_{so(3)}(0),V_{sp(2)}(0)\otimes V_{so(3)}(2\omega_1),V_{sp(2)}(\omega_1)\otimes V_{so(3)}(2\omega_1)
$$
and $V_{sp(2)}(0)\otimes V_{so(3)}(0)$ with multiplicity 2.

By searching Table 3.65 of \cite{dictionary} we see that the only possibility is that the remaining $spo(2|3)$--factors are $V_{spo(2|3)}(2\beta_1+2\beta_2)$ and $V_{spo(2|3)}(0)$, the latter with multiplicity 2.

\begin{proposition}\label{spo-simple} There is a chain of conformal embeddings
$$ V_1 (sl(2)) \otimes V_3(sl(2))  \otimes V_{-3/4} (sl(2))\!  \hookrightarrow V_1(sl(2)) \otimes  \mathcal V_{-3/4} (spo(2|3))\! \hookrightarrow V_1(G(3))\!.$$
\end{proposition} 
\begin{proof}
By Lemma \ref{singular-G3} there is a conformal embedding of 
$V_{-3/4} (sl(2))\otimes V_1(G_2)\hookrightarrow V_1(G(3))$.  By using the conformal embedding of  $V_1(sl(2)) \otimes V_3(sl(2))$ in $V_1 (G_2)$ we conclude that there is chain of conformal embeddings
$$V_1 (sl(2)) \otimes V_3(sl(2))  \otimes V_{-3/4} (sl(2))\hookrightarrow V_{1} (G_2) \otimes V_{-3/4} (sl(2))\hookrightarrow V_1(G(3)),$$ so the embedding
$$
V_1 (sl(2)) \otimes V_3(sl(2))  \otimes V_{-3/4} (sl(2))\hookrightarrow V_1(G(3)
$$
is conformal.

Since the embedding $V_1(sl(2)) \otimes  \mathcal V_{-3/4} (spo(2|3)) \hookrightarrow V_1(G(3))$ is conformal we deduce that the embedding 
$$
V_1 (sl(2)) \otimes V_3(sl(2))  \otimes V_{-3/4} (sl(2))  \hookrightarrow V_1(sl(2)) \otimes  \mathcal V_{-3/4} (spo(2|3))
$$
is conformal as well.
\end{proof}

 \begin{theorem} \label{spo-G3}Let $\be_1$, $\be_2$ be the simple roots for the distinguished set of positive roots for $spo(2|3)$. Then
 $$ V_1 (\g) = V_1 (sl(2))\otimes V_{-3/4}(spo(2|3))  \oplus L_{sl(2)}(\omega_1)\otimes L_{spo(2|3)}(V_8) . $$
 where $V_8$ is the unique irreducible $8$-dimensional representation of $spo(2|3)$ (see \cite{dictionary}).
 \end{theorem} 
 
 \begin{proof}
 Let $h_{\l,\mu}$ be the conformal weight of the highest vector of $L_{sl(2)}(\l)\otimes L_{spo(2|3)}(\mu)$.  It turns out that $h_{\l,\mu}$ with $V_{sl(2)}(\l)\otimes V_{spo(2|3)}(\mu)$ occuring in the tensor product $\g^1\otimes \g^1$ is a positive integer only in the following cases
$$
\begin{cases}
\l=0;\ \mu=0&h_{\l,\mu}=0,\\
\l=0;\ \mu=2\be_1+3\be_2&h_{\l,\mu}=0,\\
\l=0;\ \mu=\be_1+3\be_2&h_{\l,\mu}=6.
\end{cases}
$$

The only primitive vector in $V_1(G(3))$ with conformal weight $0$ is $\vac$, so, in order to apply Theorem \ref{general} we are reduced to check that there is no $\widehat \g^0$--primitive vector in $V_1(G(3))$ having conformal weight $6$.
 By Proposition  \ref{spo-simple}, there is a conformal embedding of
$V_1(sl(2))\otimes V_{3}(sl(2))\otimes V_{-3/4}(sl(2))$ in $\mathcal  V_1(\g^0)$.
We next display the possible conformal weights of $V_k(sl(2))$--singular vectors for $k=1,3,-3/4$: 
$$
\begin{tabular}{ccc}
k&\quad&\text{conformal weights}\\
\hline
1&\quad&0,\,1/4\\
3&\quad&0,\,3/20,\,2/5,\,3/4\\
-3/4&\quad&0,\,3/5,\,8/5,\,3
\end{tabular}
$$
 One cannot obtain $6$ as a sum of these values. This concludes the proof.
\end{proof}

 We are now ready to prove Theorem \ref{G2inG(3)}. The proof follows from Theorem \ref{spo-G3} and Proposition \ref{osp(3|2)-dec} by essentially repeating  the argument of \cite[Proposition 6.3]{C-RIMS}.
\begin{proof}[Proof of Theorem \ref{G2inG(3)}]
By Lemma \ref{singular-G3}, $V_1(G(3))$ is completely reducible as a $V_1(\g_{\bar 0})$ module. Thus we can write
$$
V_1(\g)=V_{-3/4}(sl(2))\otimes V_1(G_2)\oplus \sum_{\l,\mu}m_{\l,\mu}L_{sl(2)}(\l)\otimes L_{G_2}(\mu)
$$
with $\l\in\{0,\omega_1,2\omega_1,3\omega_1\}$ and $\mu\in\{0,\omega_1\}$.
Since the conformal weight of the highest weight vector of $L_{G_2}(\omega_1)$ is $2/5$, we see that $m_{\l,\mu}=0$ except when 
$$
(\l,\mu)\in\{(3\omega_1,0),(\omega_1,\omega_1),(2\omega_1,\omega_1)\}.
$$
We now check that in these cases $m_{\l,\mu}=1$. We have \cite{KS}:
$$
V_1(G_2)=V_{1}(sl(2))\otimes V_{3}(sl(2))\oplus L_{sl(2)}(\omega_1)\otimes L_{sl(2)}(3\omega_1)
$$
while, as $V_{1}(sl(2))\otimes V_{3}(sl(2))$--module,
$$
 L_{G_2}(\omega_1)=V_{1}(sl(2))\otimes L_{sl(2)}(2\omega_1)\oplus L_{sl(2)}(\omega_1)\otimes L_{sl(2)}(\omega_1),
$$
so
\begin{align*}
&V_1(\g)=\\&V_{-3/4}(sl(2))\otimes( V_{1}(sl(2))\otimes V_{3}(sl(2))\oplus L_{sl(2)}(\omega_1)\otimes L_{sl(2)}(3\omega_1))\\
&\oplus 
m_{3\omega_1,0}L_{sl(2)}(3\omega_1)\otimes (V_{1}(sl(2))\otimes V_{3}(sl(2))\oplus L_{sl(2)}(\omega_1)\otimes L_{sl(2)}(3\omega_1))\\
&\oplus 
m_{\omega_1,\omega_1}L_{sl(2)}(\omega_1)\otimes (V_{1}(sl(2))\otimes L_{sl(2)}(2\omega_1)\oplus L_{sl(2)}(\omega_1)\otimes L_{sl(2)}(\omega_1))\\
&\oplus 
m_{2\omega_1,\omega_1}L_{sl(2)}(2\omega_1)\otimes (V_{1}(sl(2))\otimes L_{sl(2)}(2\omega_1)\oplus L_{sl(2)}(\omega_1)\otimes L_{sl(2)}(\omega_1)).
\end{align*}

As $V_{-3/4} (sl(2)) \otimes V_3 (sl(2) )$--module,
$$
L_{spo(2|3)}(\be_1+3/2\be_2)=\sum c_{i,j}L_{sl(2)}(i\omega_1)\otimes L_{sl(2)}(j\omega_1),
$$
with $0\le i,j\le3$.
Since the highest weight vectors occuring in $L_{sl(2)}(\omega_1)\otimes L_{spo(2|3)}(\be_1+3/2\be_2)$ must have integral conformal weight, we have that $c_{i,j}=0$ unless $(i,j)\in\{(1,1),(2,1),(0,3),(3,3)\}$.

Combining Theorem \ref{spo-G3} and Proposition \ref{osp(3|2)-dec}, we obtain
\begin{align*}
V_1(\g)&=V_1 (sl(2))\otimes (  V_{-3/4} (sl(2))  \oplus     L_{sl(2)} (3 \omega_1)  )\otimes V_3 (sl(2))\\  
&\oplus V_1 (sl(2))\otimes (  L_{sl(2)} ( \omega_1) \oplus  L_{sl(2)} (2 \omega_1)  ) \otimes L_{sl(2)}  (2\omega_1)\\
&\oplus L_{sl(2)}(\omega_1)\otimes (c_{0,3}V_{-3/4}(sl(2))\oplus c_{3,3} L_{sl(2)}(3\omega_1))\otimes L_{sl(2)}(3\omega_1)\\
&\oplus L_{sl(2)}(\omega_1)\otimes (c_{1,1}L_{sl(2)}(\omega_1)\oplus c_{2,1} L_{sl(2)}(2\omega_1))\otimes L_{sl(2)}(\omega_1).
\end{align*}
Comparing coefficients we obtain the result.
\end{proof}
 
 \begin{remark}
 As a byproduct of the above proof we also obtain that, as a $V_{-3/4} (sl(2)) \otimes V_3 (sl(2) )$--module, 
\begin{align*}
L_{spo(2|3)}(\be_1+3/2\be_2)&=(V_{-3/4}(sl(2))\oplus L_{sl(2)}(3\omega_1))\otimes L_{sl(2)}(3\omega_1)\\
&\oplus(L_{sl(2)}(\omega_1)\oplus L_{sl(2)}(2\omega_1))\otimes L_{sl(2)}(\omega_1).
\end{align*}
 \end{remark}
 
 \section{Free field realization of $osp(m|2n)$: a new approach}\label{FFR}
 
In this section we show that the free field realization of $osp(m|2n)$, $n>0$, given in \cite{Kw} fits nicely in the general theory of conformal embeddings. Here we provide a proof based on a fusion rules argument.

Consider the superspace $\C^{m|2n}$ equipped with the standard supersym\-metric form $\langle\cdot,\cdot\rangle_{m|2n}$ given in \cite{KacReps} (sometimes denoted by  $\langle\cdot,\cdot\rangle$ if $m,n$ are clear from the context) . Let $V=\Pi \C^{m|2n}$, where $\Pi$ is the parity reversing functor. Let $M_{(m|2n)}$ be the universal vertex algebra generated by $V$ with $\l$--bracket
\begin{equation}\label{lambdaproduct}
[v_\lambda w]=\langle w,v\rangle.
\end{equation}
Let $\{e_i\}$ be the standard basis of $V$ and let $\{e^i\}$ be its dual basis with respect to $\langle \cdot,\cdot\rangle$ (i. e. $\langle e_i,e^j\rangle=\d_{ij}$). In this basis the $\lambda$-brackets are given by 
$$
[{e_h}_\lambda e_{m-k+1}]=\delta_{hk},\quad
[{e_{m+i}}_\lambda e_{m+2n-j+1}]=-\delta_{ij},\quad
[{e_{m+n+i}}_\lambda e_{m+n-j+1}]=\delta_{ij},
$$
for $h,k=1,\ldots m,\,\,i,j= 1\ldots,n$.

In the case $m= 0$ (resp. $n=0$), we write $M_{(n)} := M_{(0|2n)}$ (resp. $F_{(m/2)}:= M_{(m|0)}$. This notation is consistent with  those used in \cite{AKMPP-Selecta} and  \cite{AKMPP-IMRN}. Clearly, we have the isomorphism:
$$  M_{(m|2n)} \cong F_{(m/2)} \otimes M_{(n)}. $$

\begin{proposition}\label{EmdedFree}
  There is a non-trivial homomorphism $$\Phi:V^{1} (osp(m|2n)) \rightarrow M_{(m|2n)}$$ uniquely determined by
\begin{equation}\label{embedospM}
X\mapsto 1/2\sum_{i} :X(e_i)e^i:,\quad X\in osp(m|2n).
\end{equation}
\end{proposition}
\begin{proof}Recall that the $\lambda$--bracket of $V^1(osp(m|2n))$ is given by
$$
[X_\lambda Y]=[X,Y]+\tfrac{1}{2}\lambda\, str(XY).
$$

A straightforward computation using Wick formula shows that, if $X\in osp(m|2n)$ and $v\in V$, 
\begin{equation}\label{actiononv}
[\Phi(X)_\l v]=X(v).
\end{equation}

Applying \eqref{actiononv} and the Wick formula one obtains that
\begin{align*}
[(\tfrac{1}{2}&\sum_i:X(e_i)
e^i:)_\l (\tfrac{1}{2}\sum_i:Y(e_i)
e^i:)]=\tfrac{1}{2}\sum_{i}:[X,Y](e_i)e^i:+\tfrac{1}{2}\l\,\, str(XY).
\end{align*}
\end{proof}

A Virasoro vector for $M_{(m|2n)}$ is
$$
\omega=\tfrac{1}{2}\sum_i :T(e_i)e^i:.
$$
If $m\ne 2n+1$, let $\omega_{osp(m|2n)}$ be the Virasoro vector of $V^1(osp(m|2n))$ given by the Sugawara construction.
 \begin{lemma}\label{isconformal}Assume $m\ne2n+1$. Then
 $$\Phi(\omega_{osp(m|2n)})=\omega.$$
 \end{lemma}
\begin{proof}
It is well known that $M_{(m|2n)}$ is simple, so it is enough to show that $v(n)(\omega-\Phi(\omega_{osp(m|2n)}))=0$ for all $n>0$.

Since $[v_\l \omega]=\tfrac{1}{2}\l v$ for all $v\in V$, we need only to show that $$v(n)\Phi(\omega_{osp(m|2n)})=\delta_{n1}\tfrac{1}{2} v$$ for all $n>0$.
Using \eqref{actiononv} and the  Wick formula we see that, for $n>0$,
\begin{equation}\label{actiononL}
v(n):\Phi(X)\Phi(Y):\,\,\,=\delta_{n1} (-1)^{p(YX)p(v)+p(YX)+p(Y)p(X)}Y(X(v)).
\end{equation}

\noindent Fix a basis $\{x_i\}$ of $osp(m|2n)$ and let $\{x^i\}$ be its dual basis (i.e. $\half str(x_ix^j)=\d_{ij}$). By \eqref{actiononL}, if $n>0$,
\begin{align*}
\sum_iv(n):\Phi(x^i)\Phi(x_i):&=\delta_{n1} (\sum_i(-1)^{p(x_i)}x_i(x^i(v)))\\&=\delta_{n1} (\sum_ix^i(x_i(v)))=\d_{n1}Cv,
\end{align*}
  where $C$ is the eigenvalue of the action of the  Casimir $\sum_i x^i x_i$ on $\C^{m|2n}$.
  
   To compute this eigenvalue assume first $m\ne 2n$.
   We observe that $$str(\sum_i x^i x_i)=(m-2n)C.$$ On the other hand 
\begin{align*}
str(\sum_i x^i x_i)&=\sum_i (-1)^{p(x_i)}str(x_ix^i)\\&=2\,\sdim(osp(m|2n))=(m-2n)(m-2n-1).
\end{align*}
It follows that $C=m-2n-1$, hence
$$
v(n)\omega_{osp(m|2n)}=\d_{n1}\frac{m-2n-1}{2(1+m-2n-2)}v=\d_{n1}\tfrac{1}{2}v.
$$

We now deal with the case $m=2n$ with a more explicit calculation: recall  that $osp(2n|2n)$ is the simple Lie superalgebra of type $C(2)$ if $n=1$ and of type $D(n,n)$ if $n>1$.  We use the description of roots given in \cite{KacReps}. We choose a set of positive roots so that $\d_1\pm\e_i$ and $\d_1\pm\d_i$ are positive roots. With this choice $\d_1$ is the highest weight of $\C^{(2n|2n)}$. Calculating explicitly $(\d_1,\d_1+2\rho)$ one finds that $C=m-2n-1=-1$ also in these cases.
\end{proof}

\begin{lemma}\label{isconformaleven}If $m>1$, the embedding of $\mathcal V_1(so(m)\times sp(2n))$ in $M_{(2n|m)}$ is conformal.
\end{lemma}
\begin{proof}
By Lemma \ref{isconformal} in case $n=0$ and in case $m=0$,
$$
\mathcal V_1(so(m)\times sp(2n))=\mathcal V_1(so(m))\otimes \mathcal V_1(sp(2n))\subset M_{(m|0)}\otimes M_{(0|2n)}=M_{(m|2n)}
$$
is a conformal embedding.
\end{proof}

By \eqref{lambdaproduct} the map  $-Id$ on $V$ induces an involution of $M_{(m|2n)}$. Let $M_{(m|2n)}=M_{(m|2n)}^+\oplus M_{(m|2n)}^-$ be the corresponding eigenspace decomposition. Since $M_{(m|2n)}$ is simple, $M_{(m|2n)}^+$ is a simple vertex algebra and $M_{(m|2n)}^-$ is a simple $M_{(m|2n)}^+$--module.

\begin{theorem}Assume $n\ge 1$.
Then the image of $\Phi$ is simple; hence there is a conformal embedding of $V_1(osp(m|2n))$ in $M_{(m|2n)}$. Moreover 
$$
M_{(m|2n)}^+=V_1(osp(m|2n)),\quad M_{(m|2n)}^-=L_{osp(m|2n)}(\C^{m|n}).
$$
so that $M_{(m|2n)}$ is completely reducible as  $V_1(osp(m|2n))$--module and the decomposition is given by
$$
M_{(m|2n)}=V_1(osp(m|2n))\oplus L_{osp(m|2n)}(\C^{(m|n)}). 
$$
\end{theorem}
\begin{proof} Recall from \eqref{dot} the definition of the dot product  of two   subspaces in a vertex algebra.
Set 
$$\mathcal V_1(osp(m|2n))=\Phi(V^1(osp(m|2n)),\quad \mathcal V_1(\C^{(m|2n)})=\mathcal V_1(osp(m|2n))\cdot \C^{(m|2n)}.$$ Clearly $\mathcal V_1(osp(m|2n))\subset M^+_{(m|2n)}$ and $\mathcal V_1(\C^{(m|2n)})\subset M^-_{(m|2n)}$.  We will show that 
\begin{equation}\label{dotsquare}
\mathcal V_1(\C^{(m|2n)})\cdot \mathcal V_1(\C^{(m|2n)})\subset \mathcal V_1(osp(m|2n)),
\end{equation} so that $U = \mathcal V_1(osp(m|2n)) \oplus \mathcal V_1(\C^{(m|2n)})$ is a vertex subalgebra of $M_{(m|2n)}$. Since this vertex subalgebra  contains all generators of $M_{(m|2n)}$, we conclude that $U = M_{(m|2n)}$. This proves the statement. 

Let us first prove the case $m=0$.
\begin{itemize}
\item Let $n=1$. In this case $osp(0|2)=sl(2)$ and $\C^{(0|2)} =V_{A_1}(\omega_1)$. By using the decomposition of $sl(2)$--modules
$$ V_{A_1} (\omega_1) \otimes V_{A_1} (\omega_1) = V_{A_1} (2\omega_1) \oplus V_{A_1} (0),$$
and the fact that a  primitive vector  vector of $sl(2)$-weight $2 \omega_1$ has conformal weight $h_{2\omega_1} = \frac{2}{3} \notin {\Z}$, we conclude that  \eqref{dotsquare} holds.

\item Let $n \ge 2$. In this case $osp(0|2n)=sp(2n)$ and $\C^{(0|2n)}=V_{C_n}(\omega_1)$. Then we use the tensor product decomposition 
$$ V_{C_n} (\omega_1) \otimes V_{C_n} (\omega_1) = V_{C_n} (2\omega_1) \oplus V_{C_n} (\omega_2)  \oplus V_{C_n} (0)$$
and the fact that primitive vectors of $C_n$-weight $2 \omega_1$ and $\omega_2$ have conformal weight
$$ h_{2 \omega_1} = \frac{2 (n+1)}{2n +1} \notin {\Z}, \quad  h_{ \omega_2} = \frac{2n}{2n+1}  \notin {\Z}, $$
to conclude that \eqref{dotsquare} holds.
\end{itemize}

Now let us consider the case $m \ge 1$. If $V$ is a $so(m)$-module and $W$ is a $sp(2n)$-module we let $V \widehat \otimes W$ be the corresponding 
$so(m)\times sp(2n)$-module. As $so(m)\times sp(2n)$--module, $$\C^{(m|2n)}=\C^m\widehat \otimes V_{sp(2n)}(0)\oplus V_{so(m)}(0)\widehat \otimes \C^{2n}$$ (here $V_{so(m)}(0)=\C$ if $m=1$), so
\begin{align*}&\C^{(m|2n)}\otimes \C^{(m|2n)}\\&=(\C^m\widehat \otimes V_{sp(2n)}(0)\oplus V_{so(m)}(0)\widehat\otimes \C^{2n})\otimes(\C^m\widehat\otimes V_{sp(2n)}(0)\oplus V_{so(m)}(0)\widehat\otimes \C^{2n})\\
&=(\C^m\otimes \C^m)\widehat\otimes V_{sp(2n)}(0)+2(\C^m\widehat\otimes \C^{2n})+V_{so(m)}(0)\widehat \otimes (\C^{2n}\otimes \C^{2n}).
\end{align*}
Let $\epsilon_i, \delta_i\in \h^*$ be as in \cite{KacReps}. Let $HW$ be the set of nonzero highest weights occurring in the decomposition of  $\C^{(m|2n)}\otimes \C^{(m|2n)}$ as $so(m)\times sp(2n)$--module. Then
$$
HW=\begin{cases}\{2\delta_1,\delta_1+\delta_2,\delta_1\}&\text{if $m=1$, $n>1$}\\
\{2\delta_1,\delta_1\}&\text{if $m=1$, $n=1$}\\
\{2\epsilon_1,-2\epsilon_1, 2\delta_1,\delta_1+\delta_2,\epsilon_1+\delta_1,-\epsilon_1+\delta_1\}&\text{if $m=2$, $n>1$}\\
\{2\epsilon_1,-2\epsilon_1, 2\delta_1,\epsilon_1+\delta_1,-\epsilon_1+\delta_1\}&\text{if $m=2$, $n=1$}\\
\{2\epsilon_1,\epsilon_1, 2\delta_1,\delta_1+\delta_2,\epsilon_1+\delta_1\}&\text{if $m=3$, $n>1$}\\
\{2\epsilon_1,\epsilon_1, 2\delta_1,\epsilon_1+\delta_1\}&\text{if $m=3$, $n=1$}\\
\{2\epsilon_1,\epsilon_1+\epsilon_2, \epsilon_1-\epsilon_2,2\delta_1,\delta_1+\delta_2,\epsilon_1+\delta_1\}&\text{if $m=4$, $n>1$}\\
\{2\epsilon_1,\epsilon_1+\epsilon_2, \epsilon_1-\epsilon_2,2\delta_1,\epsilon_1+\delta_1\}&\text{if $m=4$, $n=1$}\\
\{2\epsilon_1,\epsilon_1+\epsilon_2,2\delta_1,\delta_1+\delta_2,\epsilon_1+\delta_1\}&\text{if $m\ge5$, $n>1$}\\
\{2\epsilon_1,\epsilon_1+\epsilon_2,2\delta_1,\epsilon_1+\delta_1\}&\text{if $m\ge5$, $n=1$}\\
\end{cases}.
$$

We choose the set of positive roots in $osp(m|2n)$ so that
\begin{equation}\label{rhod}
2\rho=\sum_{i=1}^n(2n-m-2i+2)\delta_i+\sum_{i=1}^{\lfloor m/2\rfloor}(m-2i)\epsilon_i.
\end{equation}
 If $\l$ is the highest weight of a $osp(m|2n)$ composition factor of $\C^{(m|2n)}\otimes \C^{(m|2n)}$ then it must occur in $HW$.
 If $m>1$ and $\l= 2\delta_1,\delta_1+\delta_2$, then, by the first part of the proof and Lemma \ref{isconformaleven}, we see that its conformal weight computed using  $\omega_{so(m)\times sp(2n)}$ is not an integer.  If  $\l\in span(\epsilon_i)$, then
 $(\l,\l+2\rho)=(\l,\l+2\rho_0)$, hence its conformal weight is
 $$\frac{(\l,\l+2\rho)}{2(m-2n-1)}=\frac{(\l,\l+2\rho_0)}{2(m-2n-1)}\ne \frac{(\l,\l+2\rho_0)}{2(m-1)},
 $$
contradicting Lemma \ref{isconformaleven}.
If $\l=\epsilon_1+\delta_1$ then, by Lemma \ref{isconformaleven}, we must have 
$
\frac{2m-2n-2}{2(m-2n-1)}=1$, which implies $n=0$. If $m=2$ and $\l=-\epsilon_1+\delta_1$ then the conformal weight is
$ \frac{-2n+2}{2(-2n+1)}\notin\ZZ$ if $n>1$ and it is $0$ if $n=1$. Finally, if $m=1$, then the conformal weight  of the elements of $HW$ computed using $\omega_{osp(1|2n)}$ is not an integer.

\end{proof}

\section{ The conformal embedding  $so(2n+8) \times sp(2n) \hookrightarrow osp(2n+8 \vert 2n)$ at $k = -2$}
 \label{osp-conformal}

\subsection{Semi-simplicity of the embedding}

In this subsection we  prove the semi-simplicity of the embedding $so(2n+8) \times sp(2n) \hookrightarrow osp(2n+8 \vert 2n)$ at $k = -2$. The corresponding decomposition will be obtained in Subsection \ref{decom-osp-2}.

\begin{theorem} \label{simplicity-d}  
(1). The vertex algebra $\mathcal V_{-2} (\g_{\bar 0})$ is simple and isomorphic to $V_{-2} (so(2n +8)) \otimes V_1 (sp(2n))$.
 
\noindent (2). $V_{-2} (\g)$ is semi-simple as $\mathcal V_{-2} (\g_{\bar 0})$--module.
\end{theorem}
\begin{proof}
Let $\ell  = n + 4 $ and  $\mathcal R_{-2} (D_{\ell })$ be the vertex algebra defined  in \cite[Section 6.1]{AKMPP-IMRN} (denoted there by $\mathcal V_{-2} (D_{\ell })$) as the  quotient of $V^{-2}(D_{\ell})$ by the ideal generated by  singular vector $w_1$ defined by formula (23) in \cite{AKMPP}. 
Recall that highest weight $\mathcal R_{-2} (D_{\ell })$--modules in $KL_{-2}$  must have highest weight $r \omega_1$ with respect to $D_{\ell}$ where $r \in {\Z} _{\ge 0} $.  
  Let $\theta'$ be the maximal root in $sp(2n)$ and let $e_{-\theta'}$ be the root 
  vector corresponding to the root $-\theta'$. Then $V_1 (sp(2n))$ is the quotient of $V^1(sp(2n))$ by the ideal generated by singular vector $e_{-\theta'} (-1) ^2{\bf 1}$. Using the  Fock-space realization of $osp(2n+8, 2n)$  at level $k=-2$, we conclude from Proposition  \ref{homo-osp-2}  that $w_1$ and $e_{-\theta'} (-1) ^2{\bf 1}$ vanish on a certain quotient of $V^{-2} (\g)$. In particular, these vectors must vanish on the simple quotient $V_{-2} (\g)$. \par We deduce that 
there is a surjective homomorphism 
 $$\mathcal R_{-2} (D_{\ell}) \otimes V_{1} (sp(2(\ell -4))) \rightarrow \mathcal V_{-2} (\g_{\bar 0}) .$$

In order to prove that $\mathcal V_{-2} (\g_{\bar 0})  $ s simple, it suffices to prove  the vanishing of the singular vector 
\bea
  w_{\ell} =
\Big(\sum _{i=2}^{ \ell} e_{\epsilon_1 - \epsilon_i}(-1)
e_{\epsilon_1 + \epsilon_i}(-1)\Big) ^{\ell -3} {\bf 1} .\nonumber \eea
(It is proved in \cite{AKMPP-IMRN} that  $w_{\ell}$  generates a unique non-trivial ideal in $\mathcal R_{-2} (D_{\ell})$).

Denote by $h[r,s]$ the conformal weight of any $ \widehat {\g}_{\bar 0}$--singular vector $v_{r,s}$ in $V_{-2} (\g)$ of $\g_{\bar 0}$--weight $(r\omega_1, \omega_s)$. By direct calculation we see that
$$ h[r,s]:= \frac{r (2n+6 +r) + (2n+2 -s) s}{4 (n+2)}. $$
In particular:
\begin{itemize}
 \item[(1)] $h[ 2n +2 - r ,r] =  2 n + 2 - r \in {\Z}_{\ge 0} $ for every $r \in \{ 0, \dots, n  \}$, 
 \item[(2)]  $h[2 n + 2-r,r-2 ] = 2 n  +1 - r + \frac{r}{2 + n}  \notin {\Z}_{\ge 0} $  for every $r  =0, \dots, n+1$,
 \item[(3)]  $h[2 n + 2 -r,r+2] = 3 + 2 n -\frac{2 + r}{2 + n} - r \notin {\Z}_{\ge 0} $  for every $r  =0, \dots, n-1$.
  \item[(4)] $h[ r ,r] =  r   $ for every  $ r \in {\Z}_{\ge 0}$.
    \item[(5)] $h[r +1, r-1] =  r +  \frac{1 + r}{2 + n}  \notin {\Z}_{\ge 0} $  for every $r  =0, \dots,  n$.
     \item[(6)] $h[r -1, r+1] =  r -  \frac{1 + r}{2 + n}  \notin {\Z}_{\ge 0} $  for every $r  =0, \dots,  n$.
\end{itemize}

By using the tensor product decomposition of $D_{\ell}$--modules
\bea   &&V_{D_{\ell}} (\omega_1) \otimes V_{D_{\ell}} (i \omega_1)  = V_{D_{\ell}} ( (i+1) \omega_1)  \oplus V_{D_{\ell}} ((i-1) \omega_1) \oplus  V_{D_{\ell}} ( \omega_2)  \label{decomp-d} \eea and the classification of irreducible $\mathcal R_{-2} (D_{\ell})$--modules 
we get the following fusion rules for $\mathcal R_{-2} (D_{\ell})$:
\bea   &&L_{-2}(\omega_1) \times L_{-2} (i \omega_1)   \subset  L_{-2} ((i+1)\omega_1  ) + L_{-2} ((i-1)\omega_1 ) \quad (i \ge 1). \label{fusion-d}\eea

It is well known that the fusion ring for $V_{1} (C_n)$ is isomorphic to the fusion ring for $V_{n} (sl(2))$
(the so-called rank-level duality).  Note that $V_{1} (\omega_n) $ is a simple current $V_{1} (C_n)$--module.
We have the following fusion rules
\bea  L_1(\omega_1)    \times L_{1} ( \omega_i ) &=& L_{1} (\omega_{i+1}  ) + L_{1} (\omega_{i-1} ) \quad (1 \le i \le n-1),  \nonumber \\
  L_1(\omega_1)    \times L_{1} ( \omega_n ) &=& L_{1} ( \omega_{n-1} ). \label{fusion-c}\eea

Assume now that there is a $ \widehat {\g}_{\bar 0}$--singular vector $v_{2n+2, 0}$ in $V_{-2} (\g)$ of $\g_{\bar 0}$--weight $((2 n +2)\omega_1, 0)$, i.e., that $w_{\ell }  \ne 0$.  
 We prove by induction that there is a   non-trivial  $ \widehat {\g}_{\bar 0}$--singular vector $v_{2n+2-r,r}$ of weight $( (2n+2-r) \omega_1, \omega_r)$  for each  $r=1, \dots, n$. 
 Using the fusion rules described above we see that $V_{-2} (\g)$ must contain non-trivial $ \widehat {\g}_{\bar 0}$--singular vector $v_{2n+1, 1}$ of $\g_{\bar 0}$--weight $((2n +1)  \omega_1, \omega_1)$ and  conformal weight $h[2n+1,1 ]=2n+1$. This gives the  induction basis. The inductive assumption says  that there is a non-trivial singular $ \widehat {\g}_{\bar 0}$--singular vector $v_{2n+2-r, r}$ of $\g_{\bar 0}$--weight $((2 n +2-r)\omega_1, \omega_r)$ for $1 \le r \le n-1$.  Its conformal weight is $h[2n+2-r,r] = 2n+2-r$. Using fusion rules and simplicity of $V_{-2}(\g)$ we conclude that at least one  of three following  $ \widehat {\g}_{\bar 0}$--singular vectors must occur:
 $$v_{2n +1-r, r+1} ,  v_{2n+1-r, r -1}, v_{2n +3-r, r+1}.$$ 
 Since $h[2n+1-r, r-1], h[2n+3-r, r-1] $ are not integers, we deduce that  $v_{2n +1-r, r+1}$ must occur. 
This completes the induction step.

 In particular, taking $r=n$  we get a singular vector $v_{n+2,n} $ of $\g_{\bar 0}$--weight
$( (n+2)\omega_1, \omega_n)$ having conformal weight $n+2$.
 Using fusion rules again we get  a  $ \widehat {\g}_{\bar 0}$--singular vector $v_{n+1, n-1}$ of $\g_{\bar 0}$--weight $((n+1)\omega_1, \omega_{n-1})$.
But the conformal weight of this singular vector is  $$h[ n+1, n-1 ]=  1 + n - \frac{1}{2 + n}  \notin {\Z}. $$ A contradiction. This proves that $w_{\ell} = 0$ in $\mathcal V_{-2} (\g)$. Therefore
$\mathcal V_{-2} (\g_{\bar 0} ) $ is a simple vertex algebra isomorphic to 
$V_{-2} (D_{\ell}) \otimes V_1 (C_n)$.  This proves assertion (1). Claim (2) follows from the fact that $V_1 (C_n)$ is a rational vertex algebra and  that the   category $KL_{-2}$ for  the vertex algebra $V_{-2} (D_{\ell}) $  is semi--simple (cf. \cite{AKMPP-IMRN}).
\end{proof}

\subsection{Realization of $osp(2n +8 \vert 2n)$ at level $k=-2$}

Combining Theorem 7.2 (2) of  \cite{AKMPP-Selecta} with Proposition 6.1 of \cite{AKMPP-IMRN} we can construct a chain of embeddings
\begin{equation}\label{embedR}
\mathcal R_{-2}(D_m)=\mathcal V_{-1/2}(so(2m))\subset V_{-1/2}(sp(4m))\hookrightarrow M_{(0|4m)},
\end{equation}
By the Symmetric Space Theorem  (see e.g. \cite{AKMPP-Selecta}, \cite{AKMPP-Indam}) we have also the chain of embeddings
\begin{equation}\label{embedSST}
V_1(sp(2n))=\mathcal V_1(sp(2n))\subset V_1(so(4n))\hookrightarrow M_{(4n|0)}.
\end{equation}
These embeddings give rise to an   embedding $$\Phi_0 :  \mathcal R_{-2} (D_m) \otimes V_1 (sp(2n)) \rightarrow M_{(4n \vert 4m)}.$$   

Consider the superspace $\C^{0|2}\otimes\C^{2m|2n}\simeq \C^{4n|4m}$. It is equipped with the supersymmetric form $\langle v\otimes w,u\otimes z\rangle=(-1)^{p(w)p(u)}\langle v,u\rangle_{0|2}\langle w,z\rangle_{2m|2n}$. Since the form is obviously invariant for $sp(2)\times osp(2m|2n)$ we obtain an embedding 
$$sl(2)\times osp(2m|2n)\hookrightarrow osp(4n|4m)
$$
hence a homomorphism
$$
\Phi\!: \!V^{-2} (osp(2m |2n)) \to \mathcal V_1(sl(2)\times osp(2m|2n))\!\!\subset V_1(osp(4n|4m))\hookrightarrow M_{(4n   \vert 4 m)}.$$

\begin{proposition} \label{homo-osp-2}
	There exists a  vertex algebra homomorphism $$\Phi: V^{-2} (osp(2n + 8 \vert 2n)) \rightarrow M_{(4n   \vert 4n +16)}$$ such that 
	$$ \Phi (V^{-2}  ( so(2 n +8 )  \times sp(2n) ) = \mathcal  R_{-2} (D_{n+4} ) \otimes V_1 (sp(2n)). $$
	\end{proposition}
\begin{proof}
The action  of $sp(2n)$ on $\C^{0|2}\otimes \C^{0|2n}$ defines the embedding $sp(2n)\subset so(4n)$ and in turn the chain of embeddings in \eqref{embedSST}.

Likewise the action  of $so(2m)$ on $\C^{0|2}\otimes \C^{2m|0}$ defines the embedding $so(2m)\subset sp(4m)$ and the chain of embeddings in \eqref{embedR}. Thus the map $\Phi_0$ is just the restriction to $\mathcal V_1(so(2m)\times sp(2n))$ of the embedding
$$\mathcal V_1(osp(2m|2n))\subset V_1(osp(4n|4m))\subset M_{(4n|4m)}=M_{(4n|0)}\otimes M_{(0|4m)}.
$$
\end{proof}
 
 We now provide explicit formulas for the odd generators of $osp(2m|2n)$.
 
Let $\{e_j\}_{j=1,2}$ be the standard basis of $\C^{0|2}$ and $\{f_j\}_{j=1,2m+2n}$ the standard basis of $\C^{2m|2n}$. By our choice of the forms $\langle\cdot,\cdot\rangle_{r|s}$, the corresponding dual bases are, respectively, $\{e^1,e^2\}$ with $e^1=e_2$, $e^2=-e_1$ and $\{f^j\}$ with
\begin{align*}
f^j=f_{2m-j+1},\ &(j=1,\cdots,2m),\ f^{2m+j}=f_{2m+2n-j+1},\ (j=1,\cdots,n),\\
& f^{2m+n+j}=-f_{2m+n-j+1},\ (j=1,\cdots,n).
\end{align*}
Let $E_{i,j}$ be the elementary matrix in the chosen basis $\{f_j\}$ of $\C^{2m|2n}$, i.e. $E_{ij}(f_r)=\d_{rj}f_i$. 
Then  $E_{i,2n+2m-j+1}-E_{2m+j,2m-i}\in osp(2m|2n)_{\bar1}$ for $1\le i\le 2m, 1\le j\le n$.

Set $v_{i,j}=e_i\otimes f_j$ and $v^{i,j}=(-1)^{p(f_j)}e^i\otimes f^j$. Clearly $\langle v_{i,j},v^{r,s}\rangle=\d_{ir}\d_{js}$.  Since $X\in osp(2m|2n)$ embeds in $osp(4n|4m)$ letting  $X$ act as $I\otimes X$ on $\C^{0|2}\otimes \C^{2m|2n}$, we obtain from \eqref{embedospM} that
\begin{align*}
&\Phi(E_{r,2n+2m-s+1}-E_{2m+s,2m-r})\\
&=1/2\sum_{ij} :(I\otimes (E_{r,2n+2m-s+1}-E_{2m+s,2m-r+1}))(v_{i,j})v^{i,j}:\\
&=1/2\sum_{i=1,2}(:v_{i,r}v^{i,2n+2m-s+1}:-:v_{i,2m+s}v^{i,2m-r+1}:).
\end{align*}

We now rewrite these odd elements in terms of the standard generators of $M_{(4n|4m)}$.
Set 
$$
\phi_i=
\begin{cases}\tfrac{\sqrt{-1}}{\sqrt{2}}(v_{1,2m+i}+v_{2, 2m+2n-i+1})&i=1,\cdots n,\\
\tfrac{1}{\sqrt{2}}(v_{1,2m+i-n}-v_{2, 2m+3n-i+1})&i=n+1,\cdots 2n,\\
\tfrac{1}{\sqrt{2}}(v_{1,2m+i-n}+v_{2, 2m+3n-i+1})&i=2n+1,\cdots 3n,\\
\tfrac{\sqrt{-1}}{\sqrt{2}}(v_{1,2m+i-2n}-v_{2, 2m+4n-i+1})&i=3n+1,\cdots 4n,
\end{cases}
$$
so that 
\begin{equation}\label{fermirel}
[{\phi_i}_\l \phi_j]=\langle \phi_j,\phi_i\rangle=\d_{ij}.
\end{equation}

Set also
$$
a^+_i=v_{2,i},\ a^-_i=v_{1,2m-i+1},\quad i=1,\ldots,2m
$$
so that 
\begin{equation}\label{weylrel}
[(a^\pm_i)_\l a^\pm_j]=0,\ [(a^+_i)_\l a^-_j]=-[(a^-_i)_\l a^+_j]=\langle a^-_j, a^+_i\rangle =\d_{ij}.
\end{equation}

Since
\begin{align*}
&v^{2,2m-r+1}=-v_{1,r}=-a^-_{2m-r+1},\\
&v^{1,2n+2m-s+1}=v_{2,2m+s}=\tfrac{1}{\sqrt{2}}(\phi_{3n-s+1}+\sqrt{-1}\phi_{4n-s+1}),
\end{align*}
and
$$
v^{1,2m-r+1}=v_{2,r}=a^+_r,\ v^{2,2n+2m-s+1}=-v_{1,2m+s}=-\tfrac{1}{\sqrt{2}}(\phi_{n+s}-\sqrt{-1}\phi_{s}),
$$
we have
\begin{align*}
&\Phi(E_{r,2n+2m-s+1}-E_{2m+s,2m-r})\\&=\tfrac{1}{2\sqrt{2}}(:a^-_{2m-r+1}(\phi_{3n-s+1}+\sqrt{-1}\phi_{4n-s+1}):-:(\phi_{n+s}-\sqrt{-1}\phi_{s})a^+_r:)\\
&-\tfrac{1}{2\sqrt{2}}(:a^+_r(\phi_{n+s}-\sqrt{-1}\phi_{s}):+:(\phi_{3n-s+1}+\sqrt{-1}\phi_{4n-s+1})a^-_{2m-r+1}:)\\
&=\tfrac{1}{\sqrt{2}}(:a^-_{2m-r+1}(\phi_{3n-s+1}+\sqrt{-1}\phi_{4n-s+1}):-:a^+_r(\phi_{n+s}-\sqrt{-1}\phi_{s}):).
\end{align*}

Recall that, if $1\le r\le m$, $1\le s\le n$, then  $E_{r,2n+2m-s+1}-E_{2m+s,2m-r}$ is the root vector $x_\a$ with $\a=\e_r+\d_s$.
Set
\begin{align*}
v_i&=\Phi(x_{\e_1+\d_i})\\&=\tfrac{1}{\sqrt{2}}(:a^-_{2m}(\phi_{3n-i+1}+\sqrt{-1}\phi_{4n-i+1}):-:a^+_1(\phi_{n+i}-\sqrt{-1}\phi_{i}):).
\end{align*}

\begin{proposition} \label{vectors-wi} Set $W_i=:v_iv_{i-1}\dots v_1:$.
Then the vectors $W_i$ are singular vectors in $M_{(4n|4m)}$ for $\widehat {so(2m)\times sp(2n)}$.
\end{proposition}
\begin{proof}
We need to show that 
\begin{equation}\label{lambafin}
[(x_{\e_j-\e_{j+1}})_\l W_i],[(x_{\e_{m-1}+\e_{m}})_\l W_i],[(x_{\d_j-\d_{j+1}})_\l W_i],[(x_{2\d_n})_\l W_i]\in \l M_{(4n|4m)}
\end{equation}
and that
\begin{equation}\label{lambdaaff}
[(x_{-\e_1-\e_{2}})_\l W_i]=x_{-\e_1-\e_{2}}(0) W_i,\ [(x_{-2\d_1})_\l W_i]=x_{-2\d_1}(0) W_i.
\end{equation}

These formulas are proven by induction on $i$. The base of the induction is $i=1$, where the formulas are satisfied since $v_1$ is a  highest weight vector for the action of $so(2m)\times sp(2n)$ on $osp(2m|2n)_{\bar 1}$.

If $i>1$, then, by Wick formula
\begin{align*}
[(x_{\a})_\l W_i]&=:[(x_{\a})_\l(x_{\e_1+\d_i})]W_{i-1}:+:x_{\e_1+\d_i}[(x_\a)_\l W_{i-1}]:\\&+\int_0^\l[[(x_{\a})_\l(x_{\e_1+\d_i})]_\mu W_{i-1}]d\mu.
\end{align*}

In order to check \eqref{lambafin}, by the induction hypothesis and the fact that 
$$
[(x_{\e_j-\e_{j+1}})_\l x_{\e_1+\d_i}]=[(x_{\e_{m-1}+\e_{m}})_\l x_{\e_1+\d_i}]=[(x_{2\d_n})_\l x_{\e_1+\d_i}]=0,
$$
$$
[(x_{\d_j-\d_{j+1}})_\l x_{\e_1+\d_i}]=\d_{i,j+1}N_{\d_j-\d_{j+1},\e_1+\d_i}x_{\e_1+\d_{i-1}},
$$
we need only to show that $
:x_{-\e_2+\d_{i}}W_{i-1}:=0$ and this follows readily since $x_{\e_1+\d_{i-1}}(-1)x_{\e_1+\d_{i-1}}(-1)=0$

In order to check \eqref{lambdaaff}, by the induction hypothesis and the fact that 
$$
[(x_{-2\d_1})_\l x_{\e_1+\d_i}]=0,\ [(x_{-\e_1-\e_2})_\l x_{\e_1+\d_i}]=N_{-\e_1-\e_2,\e_1+\d_i}x_{-\e_2+\d_i}
$$
we need only to show that $
[(x_{-\e_2+\d_{i}})_\mu W_{i-1}]=0$. An easy induction on $r$ shows that $[(x_{-\e_2+\d_{i}})_\mu W_{r}]=0$ for $1\le r< i$.

It remains to show that $W_i\ne0$ in $M_{(4n|4m)}$. By the defining relations \eqref{fermirel}--\eqref{weylrel} of $M_{(4n|4m)}$, we can write
\begin{align*}
W_i&=\,\,:(a^-_m)^i(\phi_{3n-i+1}+\sqrt{-1}\phi_{4n-i+1})\dots(\phi_{3n}+\sqrt{-1}\phi_{4n}):\\
&+\sum_{j=1}^{i}:(a^-_m)^{i-j}(a^+_1)^jc_j(\phi):
\end{align*}
with $c_j(\phi)\in M_{(4n|0)}$. The result follows.
\end{proof}

\subsection{Decomposition}
\label{decom-osp-2}

Let $\omega_{sug}$ be the Sugawara Virasoro vector in\break   $\mathcal  V_{-2} (osp(2n +8 \vert 2n))\subset V_1(osp(4n \vert 4n+16))$, $\omega^1$ the Sugawara Virasoro vector in $\mathcal R_{-2} (D_{n+4})$ and $\omega^2$, the Sugawara Virasoro vector in $V_{1} (C_n)$.
We want to investigate the  embedding $$\mathcal R_{-2}(so(2(n+4))) \otimes V_{1} (sp(2n)) \hookrightarrow \mathcal V_{-2} (\g).$$
Define 
$$\Omega = \omega_{sug} - \omega^1 - \omega^2.$$
Set for shortness $\g=osp(2n +8 \vert 2n)$.
\begin{proposition} \label{non-conf} Assume that $n \ge 2$. 
\begin{itemize}
\item[(1)]  The embedding $\mathcal R_{-2} (so_{2(n+4)}) \otimes V_{1} (sp(2n)) \hookrightarrow \mathcal  V_{-2} (\g)$ is not conformal for $n \ge 2$. 
\item[(2)]  $\Omega$ is a non-trivial Virasoro vector of central charge $c=0$. 
\item[(3)] There exists a non-trivial singular vector in $\mathcal V_{-2} (\g)$ of $\g_{\bar 0}$--weight  $(0, \omega_2)$ and conformal weight $2$. 
\end{itemize}
\end{proposition}
\begin{proof}
Assume that  $\Omega = 0$ in $\mathcal V_{-2} (\g)$, so we have a conformal embedding  $\mathcal R_{-2} (so_{2(n+4)}) \otimes V_{1}(sp(2n)) \hookrightarrow \mathcal V_{-2} (\g)$. Assume that $L(i,j)$ is an irreducible highest weight $\mathcal V_{-2} (\g)$--module with $\g_{\bar 0}$--weight $(i \omega_1, \omega_j)$, where $i, j  \in {\Z}_{\ge 0}$, $1 \le j \le n$. Recalling from \eqref{rhod} the expression of $2\rho$, we compute that the conformal weight is  given by 
$$ \Delta_{i,j}=  \frac{i^2 + j  + (2n+6) i  +6j  + j (j-1) }{8} .$$
We have
\bea  &&  \Delta_{i,j}=h[i,j]  \label{equation}  \\ \iff  && \frac{   (2n+6+ i ) i  +j (j+6)  }{8} = \frac{i (2n+6 +i) + (2n+2 -j) j}{4 (n+2)} \nonumber \\
\iff && n (2n+6+ i ) i   + (n+2) j (j+6)  = 2 j (2n+2-j) \nonumber
\eea
Assume that $i \ge 1$. Since $n (2n+6 ) >  2n (n+2) \ge 2j (2n+2-j)$ for $j=0, \dots, n$ we conclude that there are no solutions of the equation (\ref{equation}).
Assume that $i=0$. We have the equation
$$ j ( (n+2) (j+6) - 2 (2n+2-j) )   = 0  \iff j    (   j  + 2 )  (n+4)= 0.$$
Therefore, the only solution of the equation (\ref{equation}) is $(i,j) = (0,0).$
But using the free-field realization it is easy to see that there exist representations of $\mathcal V_{-2} (\g)$ in $KL_{-2}$ with highest weight different from  $(0,0)$.  Therefore, the embedding   $\mathcal R_{-2} (so(2(n+4))) \otimes V_{1} (sp(2n)) \hookrightarrow \mathcal V_{-2} (\g)$ cannot be conformal. This proves assertions (1) and (2).

Let us prove assertion (3).
Since $\Omega \ne 0$, $\mathcal V_{-2} (\g)$ contains a singular vector of conformal weight $2$. The classification of $\mathcal R_{-2}(so(2(n+4))) $ and $V_{1} (sp(2n))$--modules implies that such singular vector has $\g_{\bar 0}$--weight $(i \omega_1, \omega_j)$ for certain $i  \in {\Z}_{\ge 0}$, $1 \le j \le n$.
We see that $ \Delta_{i,j} = 2 \iff i = 0, j = 2.$
 \end{proof}
 
 \begin{remark}
 We have proved in \cite{AKMPP-IMRN}, using quantum reduction,  that the vertex algebra $V_{-2}(\g)$ has a unique irreducible module in $KL_{-2}$. Note that the proof of the previous proposition gives a new proof of this result. More precisely, each irreducible $V_{-2}(\g)$--module in $KL_{-2}$ has  $\g_{\bar 0}$--highest weight $(i \omega_1, \omega_j)$. The pair $(i,j)$ satisifes  (\ref{equation}), and the calculation in the proof of Proposition \ref{non-conf}  gives that $(i,j)=(0,0)$. Therefore, $V_{-2}(\g)$ is the unique irreducible  $V_{-2}(\g)$--module in $KL_{-2}$.
 \end{remark}

\begin{lemma} \label{lema-tezina}  For $1 \le i \le  n$, the vectors $W_i$ are not $\widehat{\g}$--singular.

\end{lemma}
\begin{proof}
Assume that $W_i$ is singular for $\widehat{\g}$. Then it generates the highest weight module with highest weight
$ \lambda_i = i \varepsilon_1 + \delta_1 + \cdots + \delta_i. $
The conformal weight of $W_i$ is  
\begin{align*}  h_{\lambda_i}& = \frac{ (\lambda_i, \lambda_i + 2 \rho )} { 2 ( k + h^{\vee}) }  = \frac{ (\lambda_i, \lambda_i + 2 \rho )} { 8 }  = \frac{i^2 + i + (2n+6) i  +6i + i (i-1) }{8} \\&= \frac{  i (  i + 2n+6) }{4}. 
\end{align*} 
So $(\omega_{sug})_0 W_i = h_{\lambda_i} W_i. $
On the other hand, $W_i$ has conformal weight $i$ in $\mathcal V_{-2} (\g)  \subset M_{(4n \vert 4n +16)}$,  which is different from  $h_{\lambda_i}$, and this is a  contradiction.
\end{proof}
\begin{proposition} \label{components-identification}  
For  $1 \le i \le n$, the vectors $W_i$ are nonzero in $V_{-2} (\g)$ and 
\bea  \mathcal W_i =  \mathcal V_{-2} (\g_{\bar 0})\cdot W_i \cong L_{-2} (i \omega_1) \otimes L_1 (\omega_i).  \label{comp-simp} \eea
\end{proposition}
\begin{proof}
 Assume that there is $j \in \{1, \dots, n\}$ so that $W_j = 0$ and that there are no non-trivial  singular vectors  in $\mathcal V_{-2} (\g)$ of weight $(i \omega_1, \omega_{i})$ for $i < j$.  By using fusion rules, we conclude that $W_j$ is singular in $\mathcal V_{-2} (\g)$. This  contradicts   Lemma \ref{lema-tezina}.  Using the formulas given  in Proposition  \ref{vectors-wi} we get (\ref{comp-simp}).
\end{proof}

\begin{theorem}\label{dosp} We have the following decomposition of $V_{-2} (\g)$:
$$ V_{-2}(\g) = \bigoplus_{i = 0} ^{n} L_{-2} (i\omega_1) \otimes L_1(\omega_i).$$
\end{theorem}
\begin{proof}
By \cite{AKMPP-IMRN}, the  irreducible $V_{-2} (so(2(n+4)))$--modules in $KL_{-2}$ are
$$  L_{-2} (i\omega_1), \ i= 1, \dots, n. $$
Therefore the irreducible $\mathcal V_{-2} (\g_{\bar 0}) $--modules  which can appear in the decomposition of  $V_{-2}(\g)$  have the form
$$ V_{i \omega_1, \omega_j}:= L_{-2} (i\omega_1) \otimes L_1 (\omega_j),  \ i, j= 1, \dots, n. $$
Note also that $\mathcal W_j = V_{j \omega_1, \omega_j}$, and that  all components in the decomposition of $V_{-2}(\g)$ appear in the fusion products 
  $$\underbrace{  \mathcal W_1 \cdots \mathcal W_1}_{k  \ \mbox{times}}, $$
  where $k$ is a positive integer.
Using fusion rules we get that
$$ \mathcal W_1 \cdot  V_{i \omega_1, \omega_i} \subset V_{(i+1) \omega_1, \omega_{i+1}}  + V_{(i-1) \omega_1, \omega_{i-1}} + V_{(i+1) \omega_1, \omega_{i-1}}  + V_{(i-1) \omega_1, \omega_{i+1}}.$$
But since $h[i+1,i-1], h[i-1,i+1] \notin {\Z}$,    the components $V_{(i+1) \omega_1, \omega_{i-1}}$, $V_{(i-1) \omega_1, \omega_{i+1}}$ can not appear. Therefore we can not get components  $V_{i \omega_1, \omega_j}$ for $i \ne j$. This implies that 
$$ V_{-2}(\g) = \bigoplus_{i = 0} ^{n}  m_i L_{-2} (i\omega_1) \otimes L_1(\omega_i) $$
for certain multiplicities $m_i \in {\Z}_{\ge 0}$.
 From Proposition \ref{components-identification} we have that   $V_{-2}(\g) $ contains $W_i$, which is a  $\widehat {\g}_{\bar 0}$--singular vector of $\g_{\bar 0}$--weight $(i \omega_1, \omega_{i})$ for $1 \le i \le n$.  So, all $m_i$ are greater or equal than $1$. Clearly $m_1=1$.
Assume now that there is $j \in {\Z}_{\ge 2}$ such that $m_{j } \ge 2$ and $m_i = 1$ for $i <  j$.  Using  the fact that $V_{-2} (\g)$ is strongly generated by $\g$, we conclude that a singular vector $v_{j,j}$  of $\g_{\bar 0}$--weight $(j \omega_1, \omega_{j})$ must appear in the fusion product
$$\underbrace{  \mathcal W_1 \cdots \mathcal W_1}_{ j  \ \mbox{times}}.$$
Using the associativity of the fusion product,  we  get that 
$$ v_{j,j} \in  \mathcal W_1 \cdot  \mathcal W_{j-1}. $$
 But if $v_{j,j}$ and $W_j$ are linearly independent,  we conclude that  the component $V_{D_{n+4}}  (j  \omega_1)  \otimes V_{C_n} (\omega_{j})$ in the tensor product 
 $$ (V_{D_{n+4}}  (( j-1)  \omega_1)  \otimes V_{C_n} (\omega_{ j -1}) )  \otimes  (V_{D_{n+4}}  (\omega_1)  \otimes V_{C_n} (\omega_{1}) )  $$ 
 has multiplicity  strictly  greater than $1$. This contradicts  the fusion/tensor product decomposition rules (\ref{decomp-d})--(\ref{fusion-c}).
 So $m_j = 1$ for $j=0, \dots, n$. The claim  follows.
 \end{proof}

The case $n=1$ is slightly different. We present  a direct  proof.

\begin{theorem}\label{ddosp}
	In the case $n=1$, $V_{-2}(\g) $  is  a simple--current extension of $\mathcal V_{-2} ( \g_{\bar 0})$.
\end{theorem}
 \begin{proof}
 	By using classification of irreducible $V_{-2} (so(10))$--modules from \cite{AKMPP-IMRN}  and tensor product decomposition
 	$$V_{D_5}(\omega_1) \otimes V_{D_5}(\omega_1) = V_{D_5}(2 \omega_1) \oplus V_{D_5}(\omega_2)  \oplus V_{D_5}(0),$$ 
 	we get that $V_{-2} (\omega_1)$ is a simple--current $V_{-2} (so(10))$--module. Since $L_1 (\omega_1)$ is also a simple--current $V_1(sl(2))$--module, we get that $V_{-2} (osp(10, 2))  $ is a simple--current extension of $V_{-2} (so(10)) \otimes V_1 (sl(2))$ and that
 	$$ V_{-2}(\g) =   V_{-2} (so(10)) \otimes V_1 (sl(2))  \bigoplus L_{-2} (\omega_1) \otimes  L_1 (\omega_1),$$
hence the claim holds.
 \end{proof}

For $n \ge 2$,  $V_{-2}(\g) $  is not a simple--current extension of $\mathcal V_{-2} ( \g_{\bar 0})$. This follows from the following fusion rules:
\begin{cor}  We have the following fusion product inside of $V_{-2} (\g)$:
\bea \mathcal W_1 \cdot \mathcal W_{i} & =&   \mathcal W_{i-1}  \oplus   \mathcal W_{i+1}  \quad (1 \le i \le n-1) \nonumber \\
\mathcal W_1 \cdot \mathcal W_{n} & =&   \mathcal W_{n-1}.  \nonumber \eea
\end{cor}

Finally, our result implies the following coset realization of $V_{-2} (so(2(n+4)))$:

\begin{cor} We have
$$ V_{-2}(so(2(n+4))) \cong  \frac{osp(2n+8 \vert 2n)_{-2}}{sp(2n)_{1}} := \mbox{Com}_{V_{-2}(osp(2n+8 \vert 2n)) } (V_1(sp(2n))). $$
\end{cor}

  \vskip10pt
  \footnotesize{
  \noindent{\bf D.A.}:  Department of Mathematics, Faculty of Science,  University of Zagreb, Bijeni\v{c}ka 30, 10 000 Zagreb, Croatia;
{\tt adamovic@math.hr}
  
%\noindent{\bf V.K.}: Department of Mathematics, MIT, 77
%Mass. Ave, Cambridge, MA 02139;\newline
%{\tt kac@math.mit.edu}

\noindent{\bf P.MF.}: Politecnico di Milano, Polo regionale di Como,
Via Valleggio 11, 22100 Como,
Italy; {\tt pierluigi.moseneder@polimi.it}

\noindent{\bf P.P.}: Dipartimento di Matematica, Sapienza Universit\`a di Roma, P.le A. Moro 2,
00185, Roma, Italy; {\tt papi@mat.uniroma1.it}

\noindent{\bf O.P.}:  Department of Mathematics, Faculty of Science, University of Zagreb, Bijeni\v{c}ka 30, 10 000 Zagreb, Croatia;
{\tt perse@math.hr}

}
\end{document}